\documentclass[11pt, reqno]{amsart}
\usepackage{enumitem,mathrsfs}
\makeatletter
\newcommand{\mylabel}[2]{#2\def\@currentlabel{#2}\label{#1}}
\makeatother
\usepackage{array}
\usepackage[english]{babel}
\usepackage[top=1in, bottom=1in, left=1in, right=1in]{geometry}

\expandafter\let\csname ver@amsthm.sty\endcsname\relax

\usepackage{tikz}

\usepackage{amsmath}
\usepackage{amssymb}
\usepackage{mathdots}
\usepackage{mathtools}

\usepackage{dsfont}
\usepackage{chngpage}

\usepackage{hyperref}
\usepackage{amsthm}
\usepackage[capitalize,noabbrev]{cleveref}

\allowdisplaybreaks

\numberwithin{equation}{section}

\newtheorem{thm}{Theorem}[section]
\newtheorem{lemma}[thm]{Lemma}
\newtheorem{cor}[thm]{Corollary}

\newtheorem{Example}[thm]{Example}

\newtheorem{Remark}[thm]{Remark}

\crefname{thm}{Theorem}{Theorems}
\crefname{lemma}{Lemma}{Lemmas}
\crefname{cor}{Corollary}{Corollaries}
\crefname{prop}{Proposition}{Propositions}

\crefname{example}{Example}{Examples}
\crefname{remark}{Remark}{Remarks}

\newcommand{\emailhref}[1]{\email{\href{#1}{#1}}}

\newcommand{\Z}{\mathbb{Z}}

\newcommand{\M}{\operatorname{M}}

\newcommand{\st}{\operatorname{s}}
\newcommand{\en}{\operatorname{e}}

\title[Squarishness, bijections and independence]{Perfect matchings and spanning trees: squarishness, bijections and independence}

\author[Seok Hyun Byun]{Seok Hyun Byun}\emailhref{sbyun@clemson.edu}
\address{School of Mathematical and Statistical Sciences, Clemson University, Clemson, South Carolina 29631, U.S.A.}

\author[Mihai Ciucu]{Mihai Ciucu}\emailhref{mciucu@indiana.edu}
\address{Department of Mathematics, Indiana University, Bloomington, Indiana 47405, U.S.A.}
\thanks{S.H.B. is supported by the AMS-Simons Travel Grant.}
\thanks{M.C. was supported in part by Simons Foundation Collaboration Grant 710477.}

\begin{document}
\maketitle

\begin{abstract}
  A number which is either the square of an integer or two times the square of an integer is called squarish. There are two main results in the literature on graphs whose number of perfect matchings is squarish: one due to Jockusch (for planar graphs invariant under rotation by 90 degrees) and the other due to the second author (concerning planar graphs with two perpendicular symmetry axes). We present in this paper a new such class, consisting of certain planar graphs that are only required to have one symmetry axis. Our proof relies on a natural bijection between the set of perfect matchings of two closely related (but not isomorphic!) families of graphs, which is interesting in its own right. The rephrasing of this bijection in terms of spanning trees turns out to be the most natural way to present this result.

  The basic move in the construction of the above bijection (which we call gliding) can also be used to extend Temperley's classical bijection between spanning trees of a planar graph and perfect matchings of a closely related graph. We present this, and as an application we answer an open question posed by Corteel, Huang and Krattenthaler.

  We also discuss another dimer bijection (used in the proof of the second author's result mentioned above), and deduce from a refinement of it new results for spanning trees. These include a finitary version of an independence result for the uniform spanning tree on $\Z^2$ due to Johansson, a counterpart of it, and a bijective proof of an independence result on edge inclusions in the uniform spanning tree on $\Z^2$ due to~Lyons.
\end{abstract}

\section{Introduction}

Kasteleyn~\cite{kasteleyn1961statistics} and independently Temperley and Fisher~\cite{fisher1961statistical, temperley1961dimer} found a formula that gives the number of perfect matchings\footnote{A perfect matching of a graph is a collection of edges that covers every vertex exactly once.} of a rectangular grid graph. Namely, for any positive integers $m$ and $n$ they showed that the number of perfect matchings of a $2m\times 2n$ grid graph is given by the double product formula

\begin{equation*}
    2^{2mn}\prod^{m}_{j=1}\prod^{n}_{k=1}\left[\cos^{2}\bigg(\frac{j\pi}{2m+1}\bigg)+\cos^{2}\bigg(\frac{k\pi}{2n+1}\bigg)\right].
\end{equation*}

One interesting fact about this formula, which is not obvious from the formula itself, is that when $m=n$ (i.e. the graph is a square grid of even side), the expression is always either the square or two times the square of an integer (throughout this paper, we call such numbers \textit{squarish}). This was first proved by Montroll using linear algebra (see \cite{lovasz2007combinatorial} for an exposition). Later, Jockusch~\cite{jockusch1994perfect} and the second author~\cite{ciucu1997enumeration} of this paper independently gave combinatorial explanations by showing that large classes of graphs with certain symmetries share the same property. More precisely, Jockusch~\cite{jockusch1994perfect} proved this for a large class of graphs invariant under rotation by $90^{\circ}$. On the other hand, as an application of the factorization theorem for perfect matchings ~\cite[Theorem 2.1]{ciucu1997enumeration}, the second author proved it for bipartite planar graphs invariant under reflection and rotation by $180^{\circ}$ (equivalently, under reflection across two perpendicular symmetry axes).


\begin{figure}[t]
\vskip0.2in
\centerline{
\hfill
{\includegraphics[width=0.8\textwidth]{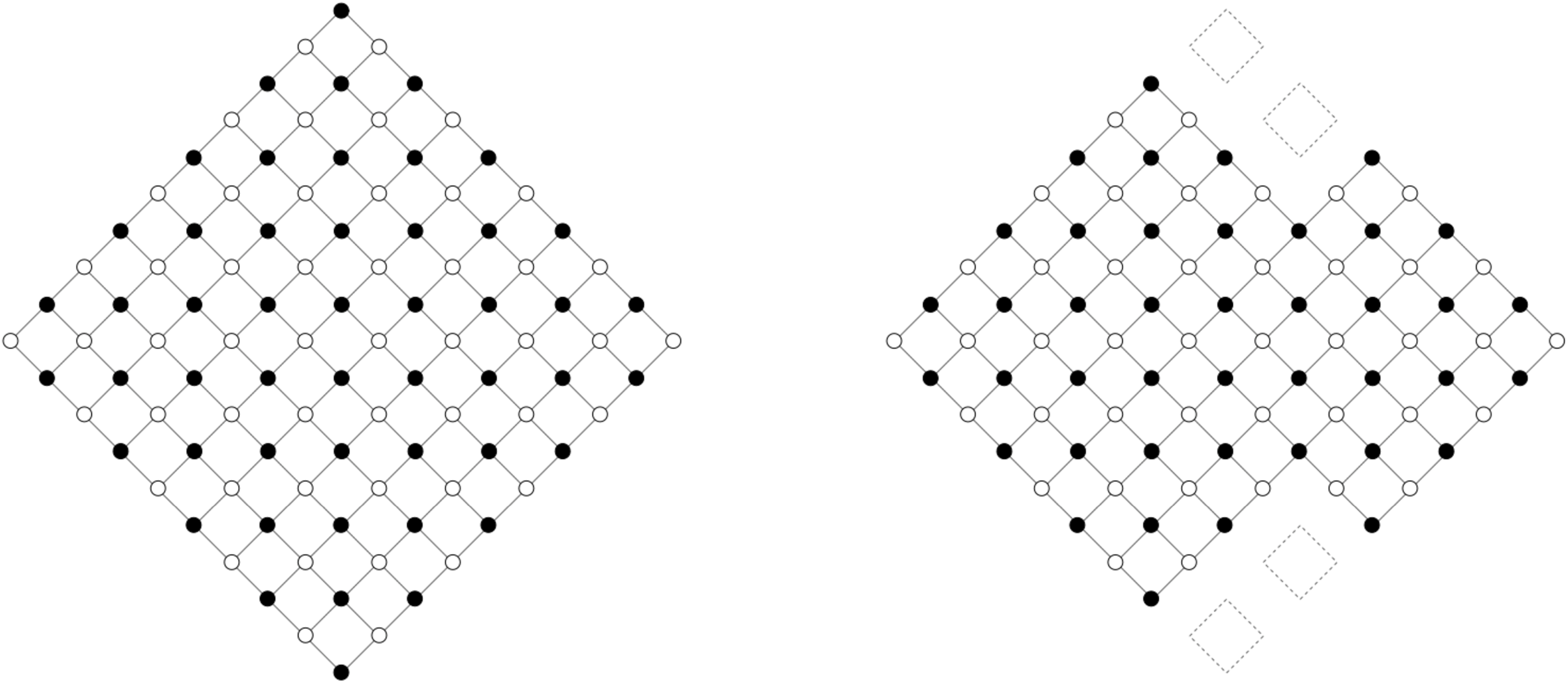}}
\hfill
}
\caption{\label{faa} The square grid graph $S_{10}$ (left) and the same graph with two symmetric pairs of four-cycles removed from opposite corners (right).}
\end{figure}

The starting point of this paper was the following observation. Let $S_{2n}$ be the $2n\times2n$ grid graph (see the picture on the left in Figure \ref{faa} for an example),
and draw it so that its diagonal $d$ is horizontal. Then, in the portion strictly above $d$, remove recursively quadruples of vertices (and all edges incident to them) that form a four-cycle, with the only requirement that one of the vertices from each removed quadruple is a ``peak'' on the boundary of the graph left over at that stage (one possible result is the top portion of the graph shown on the right in Figure \ref{faa}). 
We then remove also the reflections across $d$ of the four-cycles removed above, so that the resulting graph is symmetric about $d$ (an example is shown in the picture on the right in Figure~\ref{faa}). Numerical data suggests that the number of perfect matchings of the resulting graph is always squarish, regardless of the number of four-cycles we removed. Note that the previous results of Jockusch and the second author cannot explain this observation, because of the lack of sufficient symmetries of the graph.

The present paper consists of three inter-related parts.
In the first part, consisting of Sections 2-4, we introduce a new family of reflectively symmetric graphs (which includes the graphs described in the previous paragraph) and show that the number of perfect matchings of each of them is squarish (see Theorem \ref{tbb}). The proof is based on the factorization theorem for perfect matchings~(see~\cite[Theorem 2.1]{ciucu1997enumeration}) and a bijection between the perfect matchings of two families of closely related (but non-isomorphic!) graphs (see Theorem \ref{tba}). Using Temperley's correspondence between perfect matchings and spanning trees \cite{Temperley,lovasz2007combinatorial}, it turns out that our bijection from Theorem \ref{tba} finds its most natural phrasing in terms of spanning trees (see Theorem \ref{tbc}).

The second part, presented in Section 5, gives an extension of Temperley's above-mentioned correspondence (see Theorem \ref{tec}). A consequence of it, which is interesting in itself, is that certain collections of plane graphs with consecutive runs of vertices removed from their unbounded face have the same number of perfect matchings (this is phrased as Theorem \ref{tea}). As an application, we use this to answer an open question posed recently by Corteel, Huang and Krattenthaler \cite{Corteeletal2023AztecT1} (see Corollary \ref{teb}).

In the third part (which is presented in Section 6), motivated by the application of our perfect matching bijection from Theorem \ref{tba} to spanning trees, we also recall an earlier perfect matching bijection, due to the second author (see \cite[Lemma 1.1]{ciucu1997enumeration}), and deduce from a refinement of it (see Theorem \ref{tcb}) new results for spanning trees. These include a finitary version of an independence result for the uniform spanning tree on $\Z^2$ due to Johansson (presented in Theorem \ref{tcd}), a new counterpart of it (see Theorem \ref{tce}), and a bijective proof of an independence result on edge inclusions in the uniform spanning tree on $\Z^2$ due to~Lyons (see the proof of Theorem \ref{tcf}).





We finish the paper with some concluding remarks.


\section{The squarishness theorem and two related results} \label{sec:log-concav}

Recall that for any plane graph\footnote{ Throughout this paper, the graphs we consider have no multiple edges and no loops.
For the construction of the dual refinement $H_G$, the plane graph $G$ is assumed to have no vertex of degree one inside any bounded face.} $G$, its planar dual is the graph $G^*$ whose vertices are the bounded faces of $G$, with an edge connecting two vertices precisely if the corresponding faces of $G$ share an edge.

There is a natural way to form a graph $H_G$ --- which we will call {\it the dual refinement of $G$} --- by superimposing $G$ and its planar dual $G^*$. Namely, the vertex set of $H_G$ consists of the vertices of~$G$ (we will call them {\it original vertices}), the vertices of $G^*$ (called {\it face-vertices}), and the midpoints of the edges of $G$ (called {\it edge-vertices}). For each edge $e$ of $G$, connect the corresponding edge-vertex of $H_G$ to the two original vertices that are the endpoints of $e$, and also to the face-vertices corresponding to faces of $G$ that contain~$e$. This is the edge set of $H_G$. An example is shown in Figure \ref{fba}; the picture on the right shows the dual refinement of the graph shown on the left.

\begin{figure}[t]
    \centering
    \includegraphics[width=.92\textwidth]{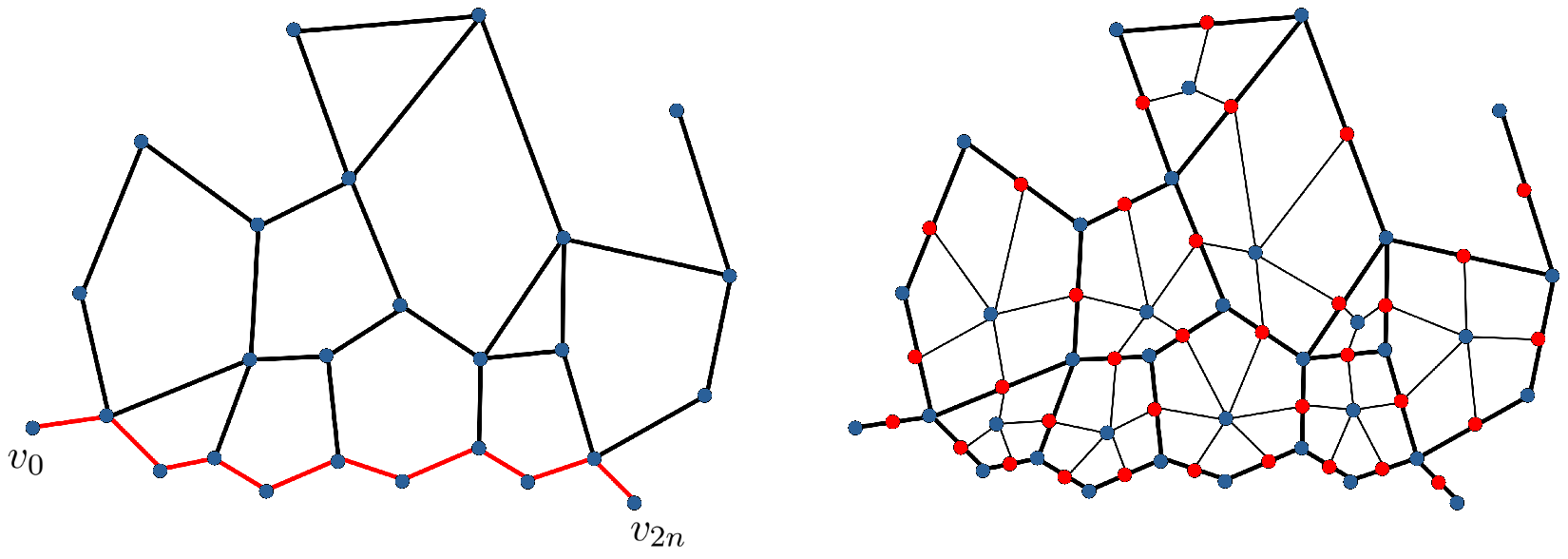}
    \caption{A plane graph $G$ (left) and its dual refinement $H_G$ (right).}
    \label{fba}
\end{figure}

Let $G_0$ be a plane graph, and let $B=v_1v_2\cdots v_{2n-1}$ be a path on the boundary of its unbounded face (we will also refer to the latter as the {\it infinite face of $G$}). Assume that all the even-indexed vertices $v_{2i}$ have degree two (i.e., there are no edges incident to them besides the edges of $B$). Denote by $G$ the graph obtained from $G_0$ by adding a leaf $v_0$ incident to $v_1$, and a leaf $v_{2n}$ incident to $v_{2n-1}$, both drawn on the infinite face (see Figure \ref{fba} for an example). Let $H_G$ be the
dual refinement of $G$
described in the previous paragraph (this is also illustrated in Figure \ref{fba}).

\begin{figure}[t]
    \centering
    \includegraphics[width=.92\textwidth]{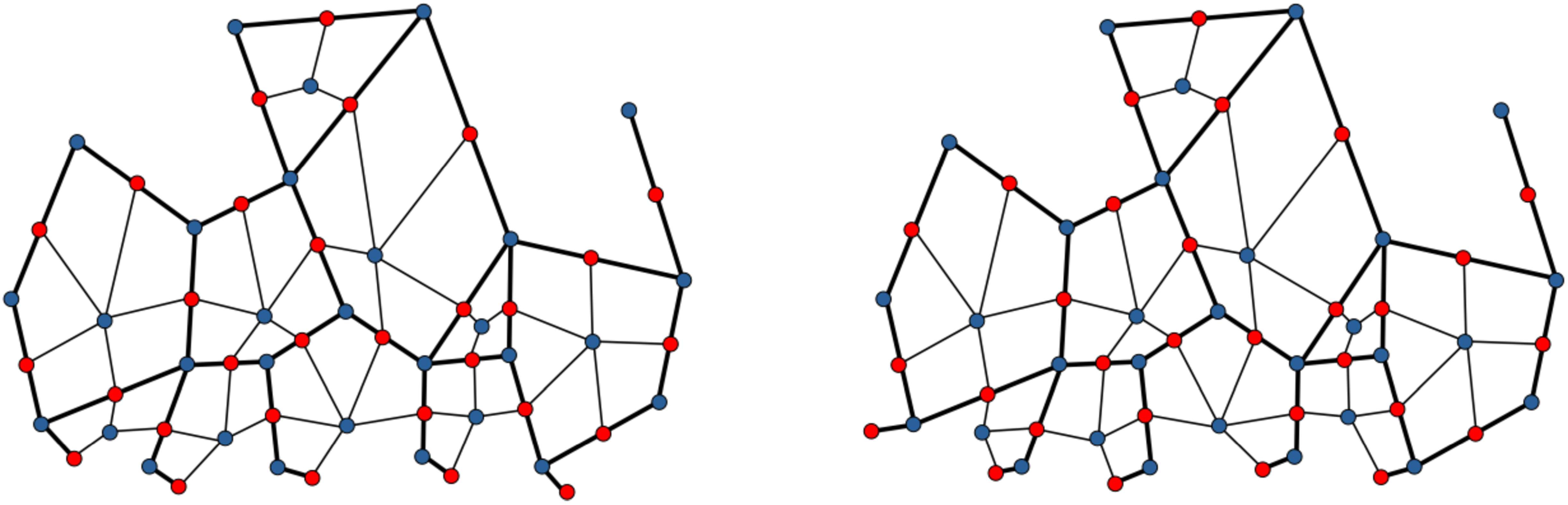}
    \caption{The graphs $G^{+}$ (left) and $G^{-}$ (right).}
    \label{fbb}
\end{figure}

\begin{figure}[t]
    \centering
    \includegraphics[width=.92\textwidth]{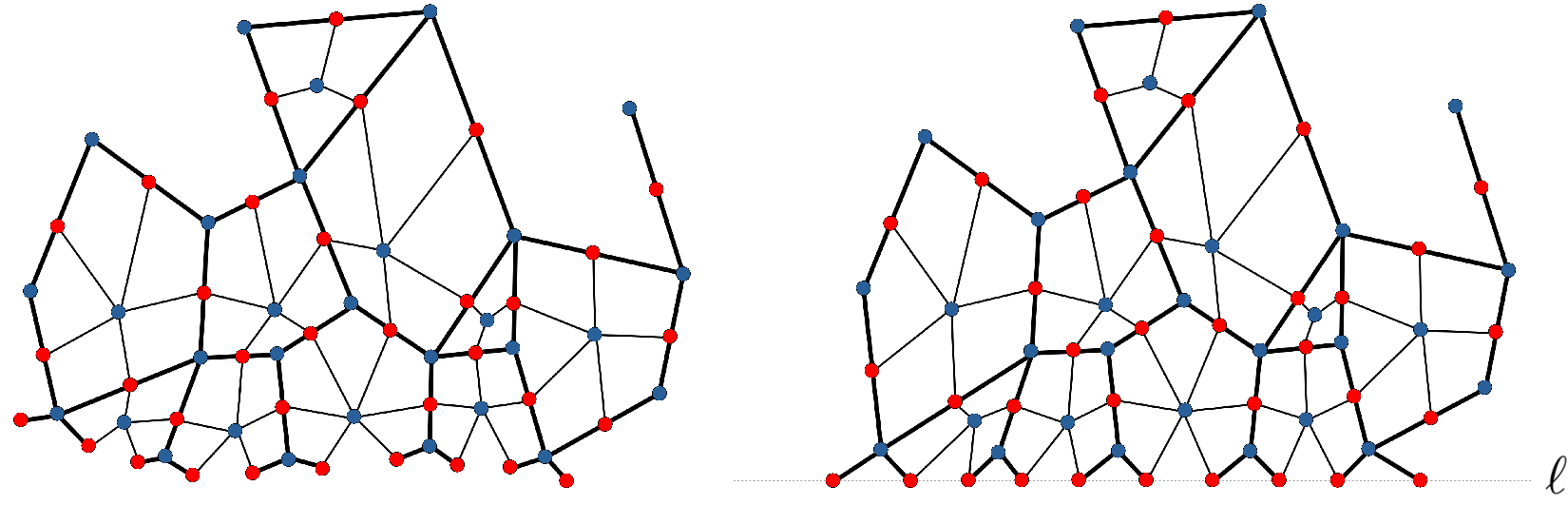}
    \caption{The graph $H_{G}\setminus\{v_{0},v_{2},\ldots,v_{2n}\}$ (left); the same graph with the $2n$ vertices $m_{1},\ldots,m_{2n}$ aligned (right).}
    \label{fbc}
\end{figure}

Let $G^+$ be graph obtained from $H_G\setminus v_{2n}$ by removing\footnote{It is a simple exercise to show that $H_G$ always has an odd number of vertices. We are interested in the perfect matchings of graphs obtained from $H_G$ by removing some of its vertices. In order for the resulting graphs to have perfect matchings, a necessary condition is to remove an odd number of vertices from $H_G$.} $v_0,v_2,\cdots,v_{2n-2}$ and their nearest neighbors towards the right on the path $B$ (see the picture on the left in Figure \ref{fbb}). Analogously, let $G^-$ be the graph obtained from $H_G\setminus v_{0}$ by removing $v_2,v_4,\cdots,v_{2n}$ and their nearest neighbors towards the left on the path $B$ (see the picture on the right in Figure \ref{fbb}; note that the nearest neighbors of $v_0,v_2,\cdots,v_{2n}$ are all edge-vertices).

We are now ready to state our first result, which will imply our squarishness theorem.

\begin{thm}
  \label{tba}
  There is a natural bijection between the sets of perfect matchings of $G^{+}$ and $G^{-}$. In particular\footnote{ $\M(G)$ denotes the number of perfect matchings of the graph $G$.}, $\M(G^{+})=\M(G^{-})$.
\end{thm}

Our family of graphs whose number of perfect matchings is squarish is obtained from $H_G$ as follows. Consider the graph $H_G\setminus\{v_0,v_2\dotsc,v_{2n}\}$ (see the picture on the left in Figure \ref{fbc}), and draw it so that its bottom vertices $($which are all edge-vertices, $2n$ in total$)$ are along a horizontal line $\ell$ (but none of its edges are along $\ell$); this is illustrated on the right in Figure \ref{fbc}.

\begin{thm}
\label{tbb} Let $\overline{G}$ be the graph obtained from $H_G\setminus\{v_0,v_2,\dotsc,v_{2n}\}$ by symmetrizing it across~$\ell$: the vertex set of $\overline{G}$ consists of the vertices of $G$ and their mirror images across $\ell$, and analogously for the edge set of $\overline{G}$ $($see the picture on the left in Figure $\ref{fbd}$ for an example$)$. Then $\M(\overline{G})$ is squarish --- i.e.\ it is either a perfect square, or two times a perfect square.
\end{thm}

\begin{figure}
    \centering
    \includegraphics[width=.9\textwidth]{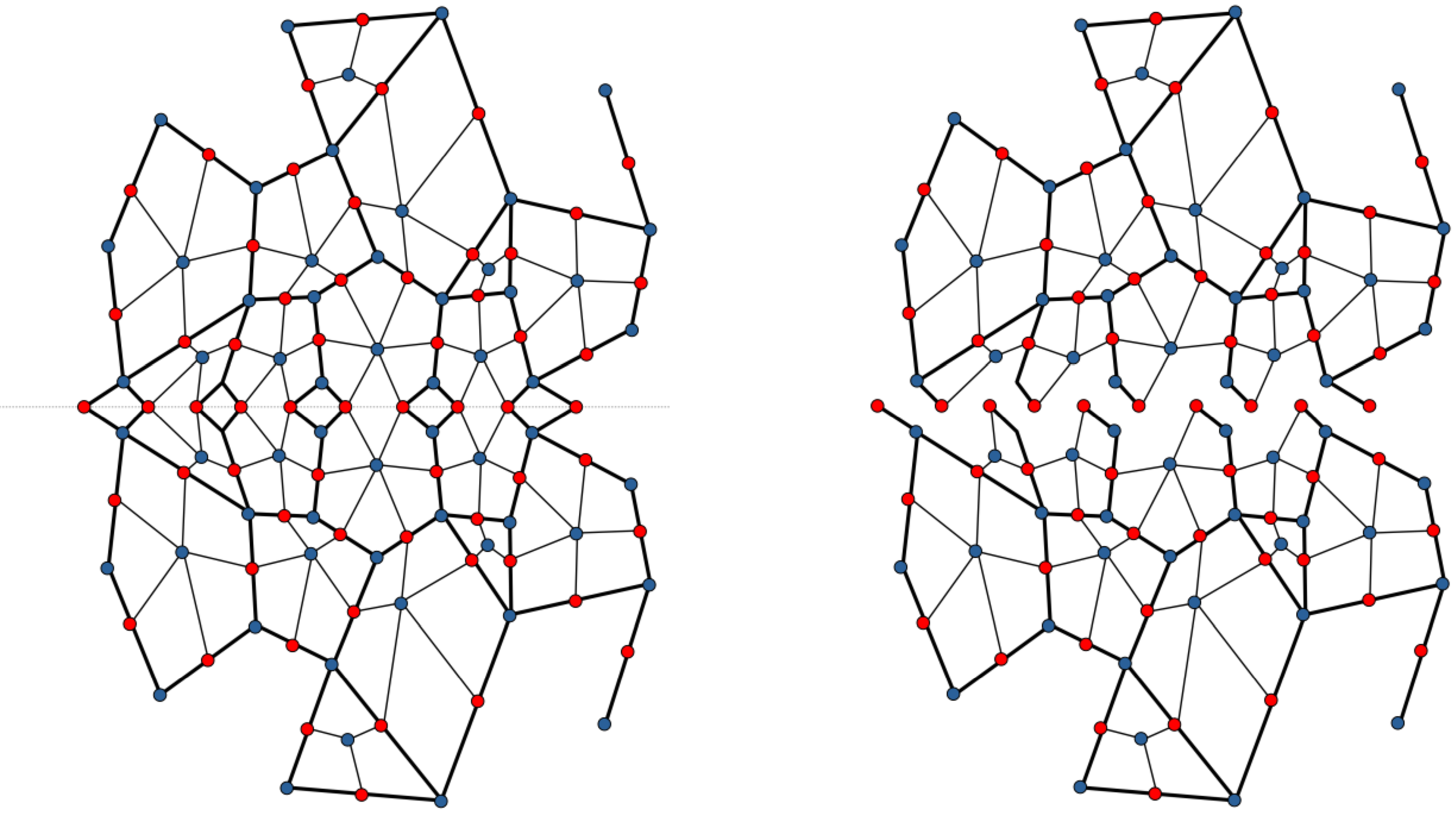}
    \caption{The graph $\overline{G}$ (left) and a decomposition of $\overline{G}$ obtained by applying the factorization theorem for matchings (right).}
    \label{fbd}
\end{figure}

\begin{proof}
  By construction, the graph $\overline{G}$ is bipartite and symmetric with respect to $\ell$. Therefore the factorization theorem for perfect matchings \cite[Theorem 2.1]{ciucu1997enumeration} can be applied to it. Since there are~$2n$ vertices on the symmetry axis, we obtain
\begin{equation}
\M(\overline{G})=2^n\M(G_1)\M(G_2),
\label{eba}
\end{equation}
where $G_1$ and $G_2$ are the portions of $\overline{G}$ above and below $\ell$ left over after deleting the edges prescribed by the factorization theorem. Since all the vertices on $\ell$ have the same color, this results in deleting alternately all the edges incident from above and all the edges incident from below, as we scan the vertices on $\ell$ from left to right. One readily sees that this implies that $G_1$ is isomorphic to our graph~$G^+$, and $G_2$ is isomorphic to the reflection across $\ell$ of $G^-$ (this can be clearly seen by comparing the pictures on the right in Figure \ref{fbd} with those in Figure \ref{fbb}). Thus equation \eqref{eba} becomes $\M(\overline{G})=2^n\M(G^+)\M(G^-)$, and the statement follows by Theorem \ref{tba}. \end{proof}

{\it Remark $1$.}
A ``leafless,'' equivalent version of Theorem \ref{tba} can readily be stated, by simply removing the leaves $v_0$ and $v_{2n}$ from $G$, and removing the forced edges from $G^+$ and $G^-$. The version stated in Theorem \ref{tba} has the advantage that symmetrizing $G\setminus\{v_0,v_2,\dotsc,v_{2n}\}$ yields a graph $\overline{G}$ which is balanced (i.e.\ has the same number of vertices in the two bipartition classes), a necessary condition for admitting perfect matchings.

\medskip
Theorem \ref{tba} and its proof give a natural bijection between the sets of perfect matchings of two graphs, $G^+$ and $G^-$, whose construction is somewhat involved. However, it turns out that the result can be equivalently rephrased in terms of spanning trees, and this rephrasing is quite natural.

\begin{thm}
\label{tbc}
Let $G$ be a plane graph and let $v_1v_2\dotsc v_{2n-1}$ be a path on the boundary of its unbounded face. Assume that vertices $v_2,v_4,\dotsc,v_{2n-2}$ have degree two. Then there is a natural bijection between the spanning trees of $G$ rooted at $v_{2n-1}$ which contain the oriented edges $(v_2,v_3),(v_4,v_5),\dotsc,(v_{2n-2},v_{2n-1})$, and the spanning trees of $G$ rooted at $v_{1}$ which contain the oriented edges $(v_2,v_1),(v_4,v_3),\dotsc,(v_{2n-2},v_{2n-3})$.

\end{thm}

The above result can be deduced from Theorem \ref{tba} using Temperley's bijection \cite{Temperley,lovasz2007combinatorial} between spanning trees of a plane graph $G$ rooted at a vertex $v$ on its infinite face and perfect matchings of the graph $H_G\setminus v$. The details are given in Section 6.

\section{The bijection $\phi:{\mathcal M}(G^+)\to{\mathcal M}(G^-)$} 

Let ${\mathcal M}(G^+)$ and ${\mathcal M}(G^-)$ be the sets of perfect matchings of the graphs $G^+$ and $G^-$, respectively.
In this section we construct a map $\phi:{\mathcal M}(G^+)\to{\mathcal M}(G^-)$ and prove that this map is well-defined and bijective, thus obtaining the proof of Theorem \ref{tba}.

The construction of the map $\phi$ is based on a certain family of non-intersecting paths $\mathcal P_\mu$ on $H_G\setminus\{v_0,v_2,\dotsc,v_{2n}\}$ that we can naturally associate to any perfect matching $\mu$ of $G^+$.

Label the edge-vertices on the bottom path $B=v_0v_1\dotsc v_{2n}$ of $G$ by $m_1,m_2,\dotsc,m_{2n}$, from left to right. Note that $G^+$ can be viewed as $H_G\setminus\{v_0,v_2,\dotsc,v_{2n}\}$ with vertices $m_1,m_3,\dotsc,m_{2n-1}$ removed, and $G^-$ can be viewed as $H_G\setminus\{v_0,v_2,\dotsc,v_{2n}\}$ with vertices $m_2,m_4,\dotsc,m_{2n}$ removed.

Let $\mu$ be a perfect matching of $G^+$. The family $\mathcal P_\mu$ will consist of $n$ non-intersecting paths on $H_G\setminus\{v_0,v_2,\dotsc,v_{2n}\}$, each connecting an $m_{2i-1}$ to an $m_{2j}$, so that all of the vertices $m_1,m_2,\dotsc,m_{2n}$ are connected up in pairs. Since $G^+$ is just $H_G\setminus\{v_0,v_2,\dotsc,v_{2n}\}$ with vertices $m_1,m_3,\dotsc,m_{2n-1}$ removed, the vertices $m_1,m_3,\dotsc,m_{2n-1}$ --- which are endpoints of distinct paths in ${\mathcal P}_\mu$ --- are not matched by $\mu$; on the other hand, the other endpoint of each of these paths --- which form the set $\{m_2,m_4,\dotsc,m_{2n}\}$ --- {\it are} matched by $\mu$.

As we will see in the paragraphs below, the paths comprising $\mathcal P_\mu$ are {\it alternating paths} with respect to $\mu$, i.e.\ their edges alternate between being contained and not being contained in $\mu$.
We define $\phi(\mu)$ to be the perfect matching of $G^-$ obtained from $\mu$ by ``shifting along the paths in ${\mathcal P}_\mu$'': along each path $P$ in ${\mathcal P}_\mu$, discard from $\mu$ the edges of $P$ that are in $\mu$, and include the edges of $P$ that are not in $\mu$; keep all the edges of $\mu$ that are not contained in any path in ${\mathcal P}_\mu$. Since the paths in $\mathcal P_\mu$ are non-intersecting, the resulting collection of edges is a perfect matching of $G^-$; we define $\phi(\mu)$ to be this perfect matching.

What remains is to define the family of non-intersecting paths $\mathcal P_\mu$ on $H_G\setminus\{v_0,v_2,\dotsc,v_{2n}\}$ corresponding to a given perfect matching $\mu$ of $G^+$.

The paths making up $\mathcal P_\mu$ are constructed in a sequence of generations of paths as follows. We call the union of the edges of the graph $G$ (viewed as sets of points in the plane) the {\it frame}, and the union of the edges of the planar dual $G^*$ the {\it dual frame}.

{\bf First generation.} To obtain the first path in the first generation, start on the left at $m_1$, and move one step
along the frame
(unique choice, because the construction takes place on the graph~$H_G\setminus\{v_0,v_2,\dotsc,v_{2n}\}$, which does not contain the other neighbor $v_0$ of $m_1$), to reach the original vertex $v_1$ (this and all the remaining part of the construction process can be followed on~Figure \ref{fpb}). Next, consider the edge $e$ of $\mu$ incident to $v_1$; note that $e$ is contained in the frame. Follow the edge~$e$ to the edge-vertex it leads to, and then ``glide'' on the frame --- i.e., continue along the edge $\overline{e}$ of $G$ containing $e$, arriving at the other endpoint of $\overline{e}$, which is also an original vertex, say $v_1'$ (this gliding operation will be the building block of our paths; it is illustrated in~Figure \ref{fpa}). From here, repeat the procedure: move along the frame to the edge-vertex to which $v_1'$ is matched under $\mu$, then glide along the frame until the next original vertex $v_1''$ is reached, and so on. We claim that this procedure can only stop when we reach an edge-vertex $m_{2i}$, and that the resulting walk $W$ is in fact a path on $H_G\setminus\{v_0,v_2,\dotsc,v_{2n}\}$ (i.e.\ all vertices of $W$ are distinct).

To see that $W$ is a path, note that by definition $W$ is contained in the frame. Then Lemma \ref{txx} implies that the original vertices $v_1,v_1',v_1'',\dotsc$ are all distinct. Indeed, otherwise they would give rise to a cycle $C$ in $G$, which by Lemma \ref{txx} has an odd number of vertices of $H_G$ in its interior, contradicting the fact that $\mu$ is a perfect matching of $G^+$. In turn, this implies that no edge-vertex on $W$ can be visited more than once (otherwise we would have to visit an original vertex on $W$ more than once). This proves that all vertices of $W$ are distinct, so $W$ is a path on $H_G\setminus\{v_0,v_2,\dotsc,v_{2n}\}$.

\begin{figure}
    \centering
    \includegraphics[width=.4\textwidth]{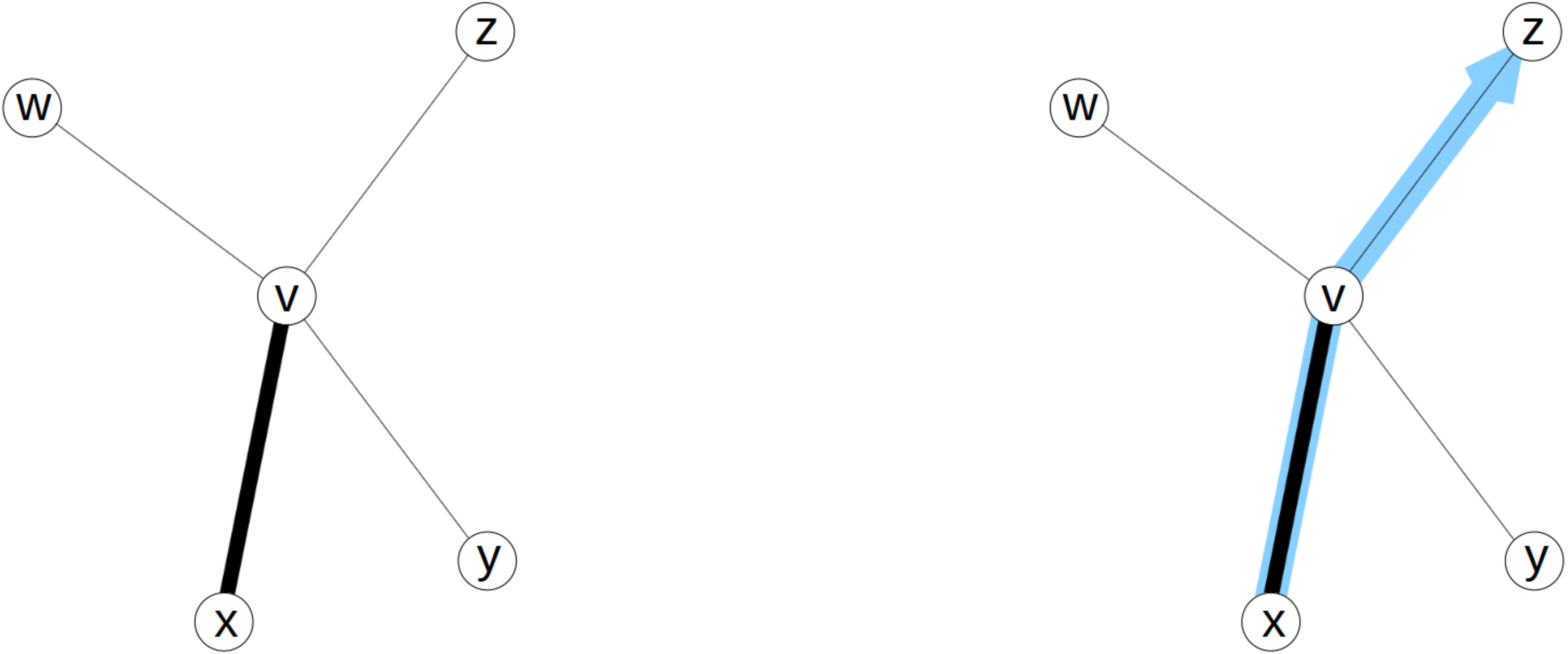}
    \caption{Gliding from $x$: Follow the edge of $\mu$ that matches $x$ (the solid line) to get to $v$; there are four edges incident to $v$; continue away from $v$ along the edge opposite the edge we followed to get to $v$ (this applies to gliding on both the frame and the dual frame). These gliding operations are the building blocks of the non-intersecting paths in ${\mathcal P}_\mu$.}
    \label{fpa}
\centerline{\includegraphics[width=0.6\textwidth]{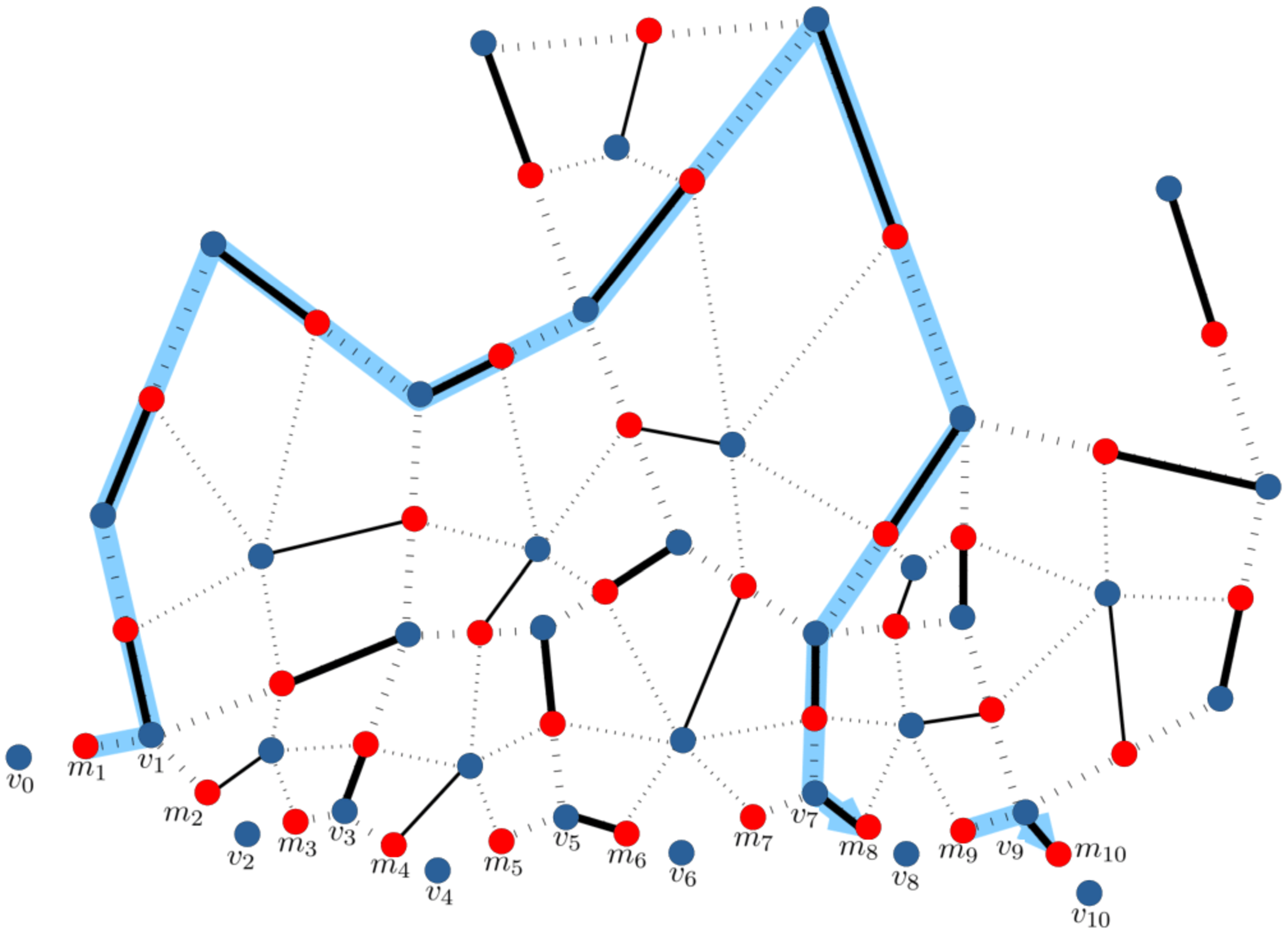}}
\centerline{\includegraphics[width=0.6\textwidth]{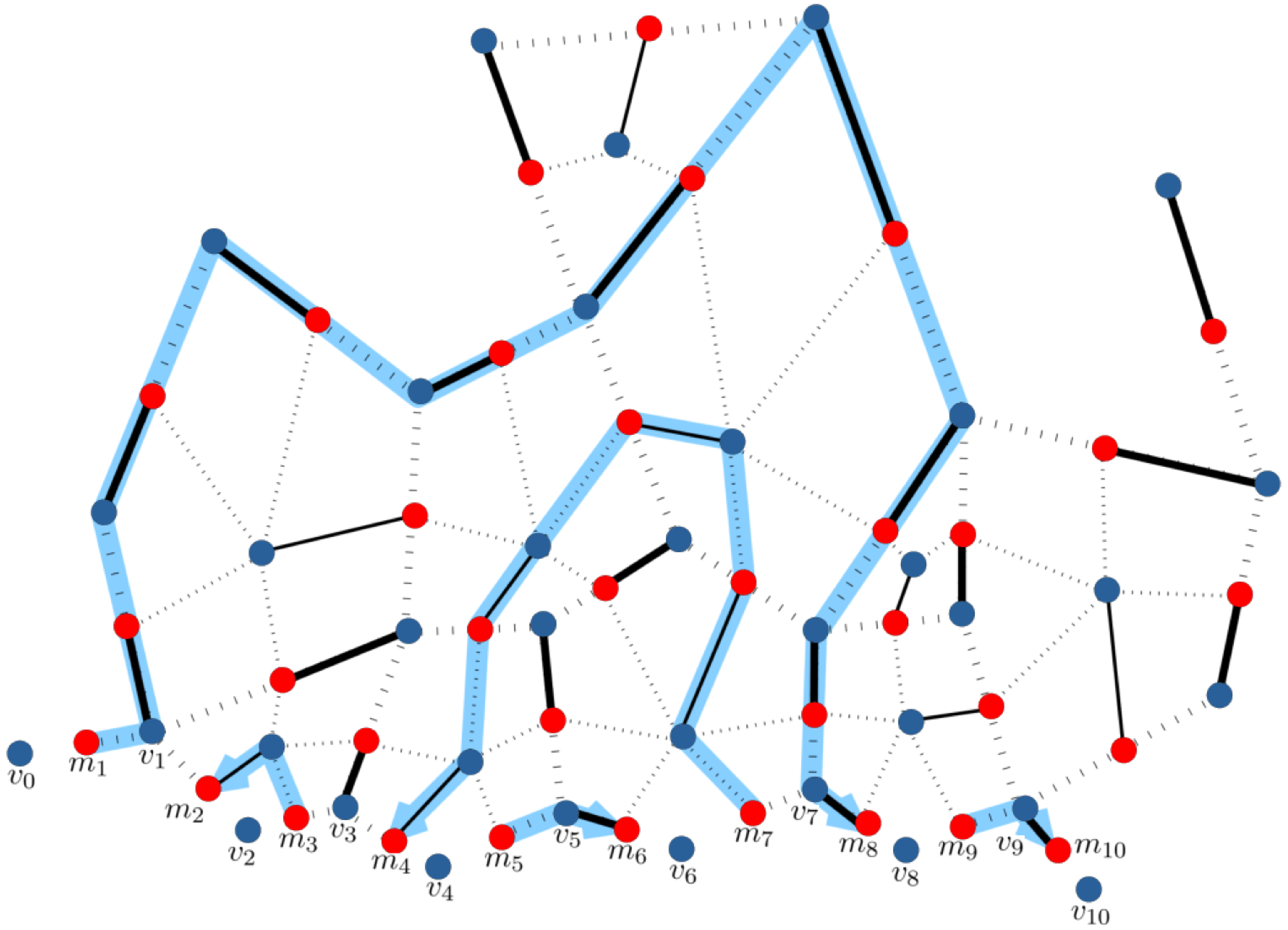}}
%
%
    \caption{{\it Top}.\ The first generation of paths. {\it Bottom}.\ The family of non-intersecting paths ${\mathcal P}_\mu$ (consisting of three generations in this example).}
    \label{fpb}
\end{figure}

To show that $W$ must end at an $m_{2i}$, let us see where could we get stuck as the path $W$ on $H_G\setminus\{v_0,v_2,\dotsc,v_{2n}\}$ is being constructed using the above-described procedure. We cannot end at an original vertex $v$ (because we follow the edge of $\mu$ incident to $v$), so we could only get stuck at an edge-vertex (since we stay on the frame). But that can only happen if we cannot glide further along on the frame from the edge-vertex we have just reached. In turn, this can only happen at an edge-vertex $m_j$ on the bottom path $B$, where gliding is blocked because the original vertex we would glide to is a $v_{2i}$, which does not belong to the graph $H_G\setminus\{v_0,v_2,\dotsc,v_{2n}\}$ on which our construction of $W$ takes place. Finally, notice that the edge-vertex $m_j$ at which we had to stop is incident to an edge of the perfect matching $\mu$ (for it was such an edge that led us to $m_j$). But only the even-indexed $m_j$'s are matched by $\mu$, so indeed we must have ended at an $m_{2i}$.




Suppose the first path in the first generation connects $m_1$ to $m_{2i}$. To obtain the second path in the first generation, apply the same procedure starting from $m_{2i+1}$ rather than from $m_1$. Repeat this until a path ends at $m_{2n}$. This is the first generation. Note that the pairs of starting and ending points of these paths are of the form $(m_1,m_{2i_1}),(m_{2i_1+1},m_{2i_2}),\dotsc,(m_{2i_k+1},m_{2n})$, with $1<i_1<\cdots<i_k<n$. We call such a system of non-intersecting paths a {\it comb}. We say that this comb {\it spans} the set $\{m_1,m_2,\dotsc,m_{2n}\}$.

It is not hard to see that all paths in the first generation are non-intersecting. Indeed, we can join together all the paths $P_1,\dotsc,P_k$ of this first generation  by concatenating consecutive ones via 2-step paths of the form $m_{2i}v_{2i}m_{2i+1}$. Since the resulting walk $P$ is on the frame, as seen in the third to last paragraph above, Lemma \ref{txx} implies that $P$ is in fact a path. This proves that all paths in the first generation are non-intersecting. 

{\bf Second generation.} To obtain the second generation of paths, proceed as follows. Under each of the paths $Q$ in the first generation, do the same construction as above, but with two modifications: (1) start from the {\it rightmost} $m_{j}$ not yet matched by our paths (which for the first path in this second generation is $m_{2i-1}$, if the path $Q$ ended at $m_{2i}$) and move one step on the {\it dual frame} (again unique choice, because the planar dual $G^*$ does not include the face-vertex corresponding to the unbounded face of $G$),
and (2) glide on the {\it dual frame} according to $\mu$. The same arguments we used above\footnote{ With Lemma \ref{txx} applied now on the planar dual graph $G^*$.} show that the paths in the second generation under $Q$ are non-intersecting and form a comb that spans the set of all $m_j$'s under $Q$. It is easy to see that each path under $Q$ in the second generation has no intersection with $Q$ (indeed, if they had a common vertex it would have to be an edge-vertex $v_e$, but then there would be two edges of $\mu$ incident to $v_e$, because $v_e$ being on $Q$ implies that an edge of $\mu$ took us to $v_e$ by gliding on the frame, while $v_e$ being on a path in the second generation implies that an edge of $\mu$ took us to $v_e$ by gliding on the dual frame). Since the paths in the first generation form a comb, the regions under different paths $Q$ are independent. Therefore, it follows that viewed together, the paths in the first and second generation are non-intersecting.

\begin{figure}
    \centering
    \includegraphics[width=.9\textwidth]{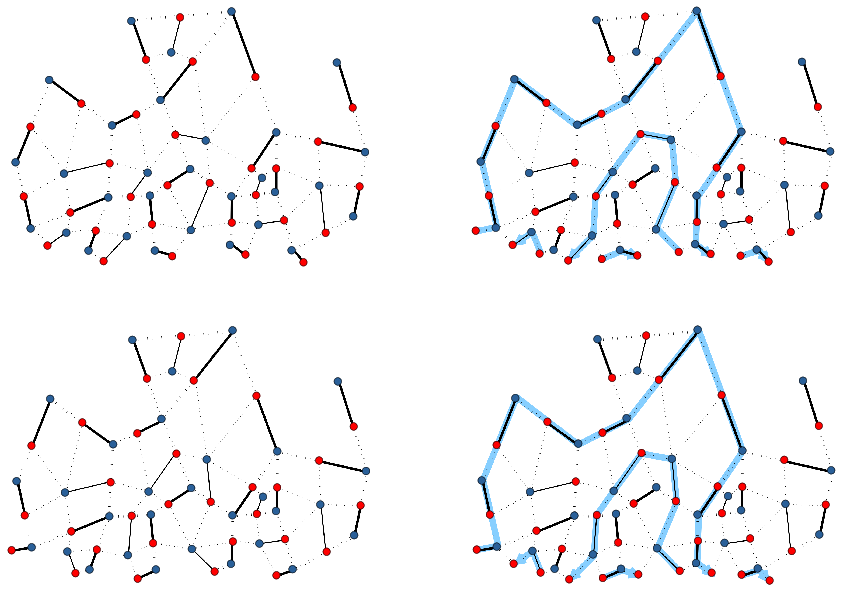}
    \caption{A perfect matching $\mu$ of $G^+$ (top left) and  the corresponding family of non-intersecting paths ${\mathcal P}_\mu$ (top right). Changing $\mu$ by shifting along the paths in~${\mathcal P}_\mu$ (bottom right), and the resulting perfect matching $\mu'=\phi(\mu)$ of $G^-$ (bottom left).}
    \label{fpc}
\end{figure}

{\bf Subsequent generations.} Subsequent generations of paths are defined following the same pattern. Namely, each subsequent odd generation is obtained by doing under each path of the generation before it exactly what we did when we constructed the first generation. Similarly, each subsequent even generation is obtained by doing under each path of the generation before it exactly what we did when we constructed the second generation. 

We continue this until all the $m_j$'s are connected up in pairs by $n$ paths. The arguments above show that these paths are non-intersecting, and connect the $m_{2i-1}$'s to the $m_{2j}$'s. We define this set of $n$ paths to be our family of non-intersecting paths ${\mathcal P}_\mu$ on $H_G\setminus\{v_0,v_2,\dotsc,v_{2n}\}$.

As explained in the fourth paragraph of this section, the map $\phi$ sends the perfect matching $\mu$ of $G^+$ to the perfect matching $\mu'$ of $G^-$ obtained from $\mu$ by shifting along the paths of $P_\mu$ (see Figure~\ref{fpc} for an illustration). This completes the definition of the map $\phi:{\mathcal M}(G^+)\to{\mathcal M}(G^-)$.

A completely analogous procedure yields a map $\psi:{\mathcal M}(G^-)\to{\mathcal M}(G^+)$, by simply swapping ``right'' and ``left'' in the above definition of $\phi$. It readily follows from our construction that if $\phi(\mu)=\mu'$, then the family of non-intersecting paths ${\mathcal P_\mu}$ along which we shift to obtain $\mu'$ from $\mu$ is the same as the family of non-intersecting paths ${\mathcal P_{\mu'}}$ along which we shift to obtain $\psi(\mu')$ from $\mu'$. Since shifting along the paths in ${\mathcal P_\mu}$ is an involution, this shows that $\psi(\phi(\mu))=\psi(\mu')=\mu$. One similarly sees that $\phi\circ\psi$ is also the identity. This implies that $\phi$ is a bijection, thus proving~Theorem~\ref{tba}.

\section{Proof of Theorem \ref{tbc}}

We first recall a very useful bijection (found by Temperley \cite{Temperley} for subgraphs of the grid graph, and generalized in Lov\'asz \cite[Problem 4.30(b)]{lovasz2007combinatorial} for arbitrary plane graphs; see also \cite{KPW}) between the spanning trees\footnote{ A spanning tree of a graph $G$ is a subgraph of $G$ which is a tree and contains all the vertices of $G$.} of a plane graph and perfect matchings of a closely related graph.

Given a spanning tree $T$ of a graph $G$ and a vertex $v$ of $G$, orient each edge of $T$ so that it points toward $v$ (more precisely, for each edge $e$ of $T$ consider the unique path $P$ in $T$ that connects the midpoint of $e$ to $v$; orient $e$ so that it points toward $v$ along the path $P$). We say that the obtained oriented tree is a {\it spanning tree of $G$ rooted at $v$}. It is not hard to see that, under this orientation, at each vertex of $T$ other than the root there is a unique outgoing edge.

\begin{thm}[\sc Temperley \cite{Temperley,lovasz2007combinatorial}]
\label{tda}
\ 
\newline
\phantom{aa}$({\rm a})$.\ Let $G$ be a plane graph and let $v$ be a vertex of $G$ on the unbounded face. Then there is a natural bijection between the spanning trees of $G$ rooted at $v$ and the perfect matchings of the graph $H_G\setminus v$, where $H_G$ is the dual refinement of $G$ defined at the beginning of Section~$2$.

$({\rm b})$.\ Assume $G$ has a weight function on its edges. This induces a weight on the edges of $H_G$ by weighting all the edges contained in edges of $G$ by the corresponding weights in $G$, and all the edges contained in edges of the planar dual $G^*$ by $1$. Then the bijection in part $({\rm a})$ is weight-preserving.

\end{thm}

\begin{proof}

Given a spanning tree $T$ of $G$ rooted at $v$, we construct the corresponding perfect matching $\mu_T$ of $H_G\setminus v$ as follows. All the edges of $T$ are on the frame (see the beginning of Section 3). Include in $\mu_T$ all the ``tail half-edges'' of $T$ (i.e., for each edge $\overline{e}$ of $T$, from among the two edges of $H_G$ contained in $\overline{e}$, include in $\mu_T$ the one which is at the tail end of $\overline{e}$); note that these are disjoint, because as we pointed out above, $T$ has a unique outgoing edge at each vertex. These are all the edges of $\mu_T$ contained in the frame.

Consider the tree $T^*$ which is the dual of $T$. Then $T^*$ is a spanning tree of the dual graph $G^*\cup z$ of~$G$ (recall that the planar dual $G^*$ of $G$ is obtained from the dual graph of $G$ by removing from the dual graph the vertex $z$ corresponding to the unbounded face of $G$); $T^*$ contains an edge $e^*$ of the dual graph $G^*\cup z$ precisely if the edge $e$ of $G$ corresponding\footnote{ I.e.\ the edge  $e^*$ is the dual of $e$.} to $e^*$ is not contained in $T$ . Orient the edges of $T^*$ so that it is rooted at $z$. Include in $\mu_T$ also the tail half-edges of $T^*$. These are also disjoint from each other; they are also disjoint from the half-edges we included from the frame, as $T$ and $T^*$ are disjoint.

Let $\nu$ be the number of vertices of $G$, $\epsilon$ the number of its edges, and $\varphi$ the number of its bounded faces. Then by Euler's theorem we have $\nu-\epsilon+\varphi=1$, so $\epsilon=\nu+\varphi-1$. Note that by the procedure in the previous paragraph, in $\mu_T$ were included $\nu-1$ edges contained in the frame (one for each edge of the spanning tree $T$) and $\varphi$ edges contained in the dual frame (one corresponding to each edge of the dual spanning tree $T^*$), for a total of $\nu+\varphi-1$ edges. Since these edges are disjoint and the number of vertices of $H_G\setminus v$ is $\nu+\epsilon+\varphi-1=\nu+(\nu+\varphi-1)+\varphi-1=2(\nu+\varphi-1)$, $\mu_T$ is indeed a perfect matching of $H_G\setminus v$.

Conversely, to construct the spanning tree $T_\mu$ of $G$ rooted at $v$ corresponding to a given perfect matching $\mu$ of $H_G\setminus v$, proceed as follows. For each vertex $u$ of $G\setminus v$, consider the edge $e$ in $\mu$ incident to $u$, and include in $T_\mu$ the edge $\overline{e}$ of $G$ containing $e$;
orient $\overline{e}$ so that $e$ becomes its tail half-edge.
We claim that $T_\mu$ is a spanning tree of $G$ rooted at $v$.

Indeed, the number of edges we included in $T_\mu$ is $\nu-1$. In addition, Lemma \ref{txx}. implies that there is no cycle formed by these edges. These two facts imply that $T_\mu$ is a spanning tree of $G$.

Note that $T_\mu$ together with the midpoints of its edges can naturally be regarded as a subgraph of~$H_G$; denote this subgraph by $t_\mu$. Then the edges of $\mu$ contained in $T_\mu$ form a perfect matching~of~$t_\mu\setminus v$ (because this is a collection of $\nu-1$ disjoint edges, and the number of vertices of $t_\mu\setminus v$ is~$2\nu-2$).

However, it is not hard to show that given any tree $T$ rooted at $u$, if $t$ is the graph obtained from $T$ by subdividing each edge into a path with two edges, then the graph $t\setminus u$ has a unique perfect matching, which consists of the tail half-edges of $T$ (this follows readily by induction on the number of edges of $T$).

This shows that, if we orient each edge of $T_\mu$ as specified above (i.e.\ so that the edge of $\mu$ contained in it becomes its tail half-edge), then $T_\mu$ becomes a spanning tree rooted at $v$, as claimed.


One readily sees from our construction of the maps $T\mapsto\mu_T$ and $\mu\mapsto T_\mu$ that they are inverses of each other. This proves that they are both bijections, and thus part $({\rm a})$.

It readily follows from our construction that, in the presence of the weight defined in part (b), these bijections are weight preserving. This completes the proof.
\end{proof}

\begin{cor}
\label{tdb}
Let $T$ be a spanning tree of $G$ rooted at $v$, and $\mu$ a perfect matching of $H_G\setminus v$, corresponding to each other under Temperley's bijection. Let $e_i=\{x_i,y_i\}$, $i=1,\dotsc,s$, be edges of~$G$. Then $T$ contains the edge $e_i$ oriented from $x_i$ to $y_i$, for $i=1,\dotsc,s$, if and only if $\mu$ contains the half-edge of $e_i$ incident to $x_i$, for $i=1,\dotsc,s$.
\end{cor}

\begin{proof} This follows directly from the construction of the maps $T\mapsto\mu_T$ and $\mu\mapsto T_\mu$ in the above proof of Theorem \ref{tda}.
\end{proof}

{\it Proof of Theorem $\ref{tbc}$.} Let $G_0$ be the graph obtained from $G$ by adding two vertices of degree one $v_0$ and $v_{2n}$, incident to $v_1$ and $v_{2n-1}$, respectively\footnote{ Note that in this proof the roles of $G$ and $G_0$ have been swapped compared to what they were at the beginning of Section 2. This is due to the fact that in Theorem \ref{tbc} the more natural notation $G$ (as opposed to $G_0$) is used for the given plane graph.} . It is clear that the spanning trees of $G$ rooted at~$v_{2n-1}$ which contain the oriented edges $(v_2,v_3),(v_4,v_5),\dotsc,(v_{2n-2},v_{2n-1})$ can be identified with the spanning trees of $G_0$ rooted at $v_{2n}$ which contain the same oriented edges, plus the edge $(v_0,v_1)$ (the latter is contained in all spanning trees of $G_0$, because $v_0$ is a leaf). Similarly, the spanning trees of $G$ rooted at $v_{1}$ which contain the oriented edges $(v_2,v_1),(v_4,v_3),\dotsc$, $(v_{2n-2},v_{2n-3})$ can be identified with the spanning trees of $G_0$ rooted at $v_{0}$ which contain these same oriented edges, and in addition also the edge $(v_{2n},v_{2n-1})$ (again, due to the fact that $v_{2n}$ is a leaf). Therefore, it suffices to show that there is a natural bijection between the two above-described sets of rooted spanning trees of~$G_0$.

As a consequence of Corollary \ref{tdb}, a spanning tree $T$ of $G_0$ rooted at $v_{2n}$ contains the oriented edges $(v_0,v_1),(v_2,v_3),(v_4,v_5),\dotsc,(v_{2n-2},v_{2n-1})$ precisely when the perfect matching $\mu_T$ of $H_{G_0}\setminus v_{2n}$ corresponding to it under Temperley's bijection contains $\{v_0,m_1\},\{v_2,m_3\},\dotsc,\{v_{2n-2},m_{2n-1}\}$; in turn, these perfect matchings $\mu_T$ can be identified with the perfect matchings of $G_0^+$. By the same argument, the spanning trees of $G_0$ rooted at $v_{0}$ which contain the oriented edges $(v_2,v_1),(v_4,v_3),\dotsc,(v_{2n-2},v_{2n-3}),(v_{2n},v_{2n-1})$ are in correspondence under Temperley's bijection with the perfect matchings of $G_0^-$. Since the perfect matchings of $G_0^+$ and~$G_0^-$ are in natural bijection by Theorem \ref{tba}, the proof is complete. $\hfill\square$

\section{An extension of Temperley's theorem and a related perfect matching bijection} 

An immediate consequence of Temperley's bijection (see Theorem \ref{tda}) is that the number of perfect matchings of the graph $H_G\setminus v$ is independent of which vertex $v$ of $G$ on the boundary of the infinite face is removed from the dual refinement $H_G$ of the graph $G$.

The first result of this section (see Theorem \ref{tea} below) is a generalization of this. We use it to answer a question posed by Corteel, Huang and Krattenthaler \cite{Corteeletal2023AztecT1} on finding a bijection between the perfect matchings of two families of graphs (see Corollary \ref{teb}). We conclude this section with an extension of Temperley's bijection (stated as Theorem \ref{tec}). While Temperley's theorem gives a bijection between the set of spanning trees of a plane graph and the set of perfect matchings of its dual refinement with a vertex removed, our extension provides a bijection between the set of spanning forests of a plane graph that satisfy certain conditions, and the set of perfect matchings of a related graph.

Let $G$ be a weighted plane graph, and consider an arbitrary weight function on the edges of its planar dual $G^*$. Let $H_G$ be the dual refinement of $G$ introduced at the beginning of Section 2. Then there is a natural induced weight on the edges of $H_G$: weight each edge of $H_G$ by the weight of the edge of $G$ or $G^*$ that contains it\footnote{ The only edges of $H_G$ that are not contained in any edge of $G$ or $G^*$ are those that connect an edge-vertex on the infinite face to a face-vertex; weight all these by 1.}.

Our results involve choosing two sets of original vertices (i.e.\ vertices of $G$) on the boundary of the infinite (i.e.\ unbounded) face of
$G$
(see the picture on the left in Figure \ref{fea}; $G$ is shown in thick lines): $v_1,v_2,\dotsc,v_{2n+1}$ and $v'_1,v'_2,\dotsc,v'_{2n+1}$, so that\footnote{ For Theorem \ref{tec} we will drop condition $(i)$.} 

\medskip
$(i)$ $v_1v_2\dotsc v_{2n+1}$ is a path in $G$

\smallskip
$(ii)$ $v'_1v'_2\dotsc v'_{2n+1}$ is a path in $G$

\smallskip
$(iii)$ vertices $v_1,v_2,\dotsc,v_{2n+1},v'_{2n+1},v'_{2n},\dotsc,v'_{1}$ are in cyclic order\footnote{ Without loss of generality, we assume throughout this section that this is the counterclockwise order.}

\smallskip
$(iv)$ vertices $v_2,v_4,\dotsc,v_{2n}$ and $v'_2,v'_4,\dotsc,v'_{2n}$ have degree 2, and belong to distinct bounded faces of $G$.

\medskip
We need one more definition before stating the first result in this section. Let $v$ be an original vertex in $H_G$ that is on the boundary of the infinite face and has degree 2. Then we say that the graph obtained from $H_G$ by removing\footnote{ When removing a vertex from a graph, all edges incident to it are also removed.} $v$ and the two edge-vertices on the infinite face adjacent to $v$ is obtained by {\it smashing in $v$} (the picture on the right in Figure \ref{fea} illustrates this construction for $v_2$, $v_4$, $v'_2$ and $v'_4$); we
say that the face-vertex corresponding to the bounded face containing $v$ was obtained by {\it smashing in vertex $v$} (in Figure \ref{fea}, $f_2$, $f_4$, $f'_2$ and $f'_4$ were obtained by smashing in vertices $v_2$, $v_4$, $v'_2$ and $v'_4$, respectively).

Let $\widehat{H}_G$ be the graph obtained from $H_G$ by smashing in vertices $v_2,v_4,\dotsc,v_{2n},v'_2,v'_4,\dotsc,v'_{2n}$ (these operations are well defined, by condition $(iv)$ above), and denote by $f_2,f_4,\dotsc,f_{2n},f'_2,f'_4,\dotsc,f'_{2n}$ the face-vertices obtained by smashing in these $v_i$'s and $v'_i$'s, respectively.


\begin{thm}
\label{tea}
$(${\rm a}$)$.\ There is a weight-preserving bijection between the set of perfect matchings of $\widehat{H}_{G}\setminus \{v_1,f_2,v_3,f_4,\dotsc,f_{2n},v_{2n+1}\}$ and the set of perfect matchings of $\widehat{H}_{G}\setminus \{v'_1,f'_2,v'_3,f'_4,\dotsc,f'_{2n},v'_{2n+1}\}$.

\medskip
$(${\rm b}$)$. More generally, let $P_1,P_2,\dotsc,P_{2n+1}$ be non-intersecting paths on $H_G$ so that $P_{2i-1}$ is a path in $G$ connecting $v_{2i-1}$ to $v'_{2i-1}$, for $i=1,\dotsc,n+1$, and $P_{2i}$ is a path in the planar dual $G^*$ connecting $f_{2i}$ to $f'_{2i}$, for $i=1,\dotsc,n$. For a path $P$, denote by $\st_P$ the first $($starting$)$ vertex in the path, and by $\en_P$ the last $($ending$)$ vertex in the path\footnote{ The paths $P_1,P_2,\dotsc,P_{2n+1}$ are regarded as starting at $v_1,f_2,v_3,f_4,\dotsc,f_{2n},v_{2n+1}$, respectively.}. 

Then for any fixed $I\subseteq[2n+1]$ there is a weight-preserving bijection between

\medskip
$(1)$ the set of perfect matchings of $\widehat{H}_{G}\setminus \{v_1,f_2,v_3,f_4,\dotsc,f_{2n},v_{2n+1}\}$ which contain all the edges of the unique perfect matching of the path $P_i\setminus \st_{P_i}$, for $i\in I$

\medskip
and

\medskip
$(2)$ the set of perfect matchings of $\widehat{H}_{G}\setminus \{v'_1,f'_2,v'_3,f'_4,\dotsc,f'_{2n},v'_{2n+1}\}$ which contain all the edges of the unique perfect matching of the
path $P_i\setminus \en_{P_i}$, for $i\in I$.

\end{thm}

\begin{figure}
    \centering
    \includegraphics[width=.9\textwidth]{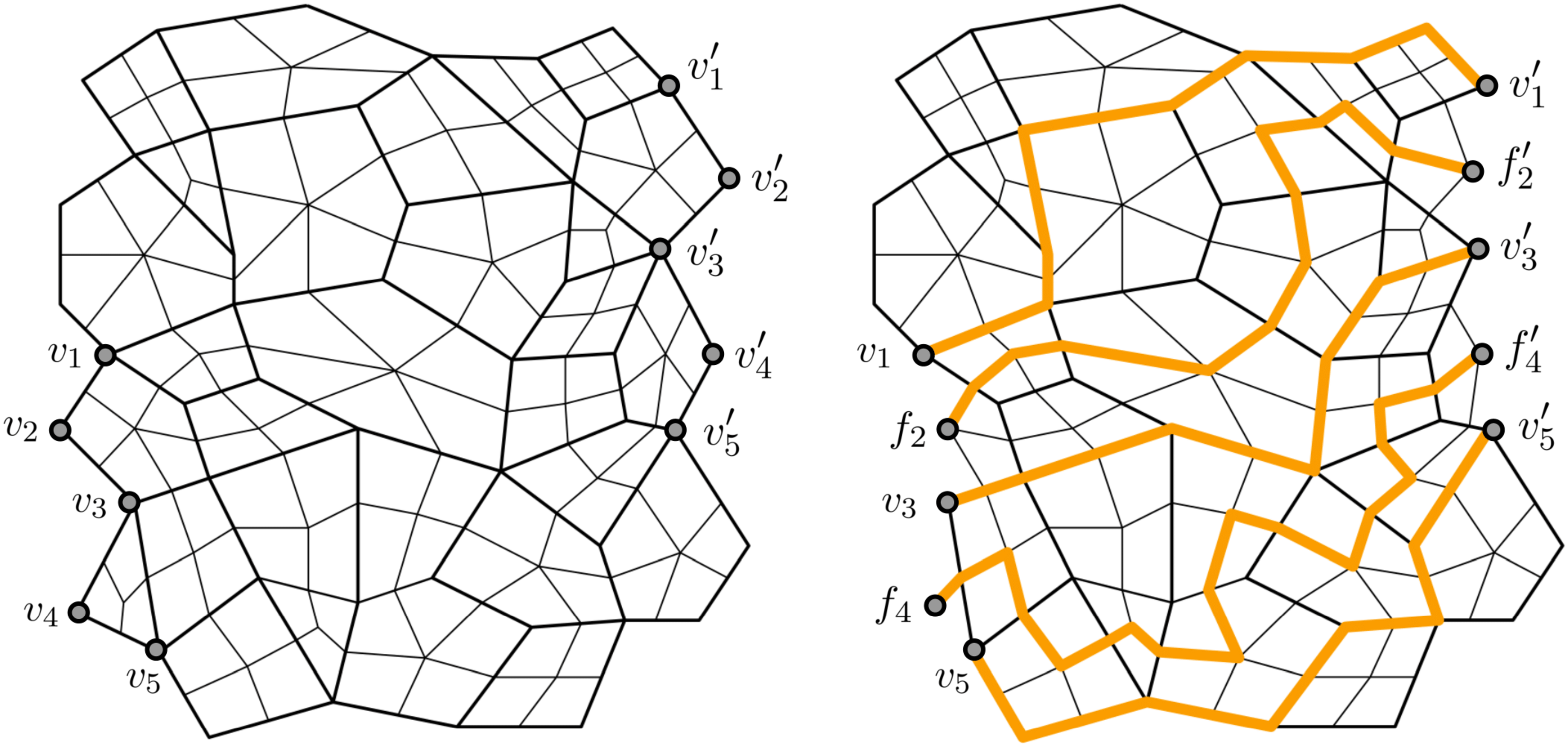}
    \caption{{\it Left.} The graph $H_{G}$ with a choice of vertices $v_{1},\cdots,v_{5}$ and $v'_{1},\cdots,v'_{5}$.
{\it Right.} The graph $\widehat{H}_{G}$ obtained from the left picture by smashing in vertices $v_2,v_4,v'_2,v'_4$; five non-intersecting paths $P_{1},P_{2},P_{3},P_{4},P_{5}$ are also shown.}
    \label{fea}
\end{figure}

\begin{figure}
    \centering
    \includegraphics[width=.9\textwidth]{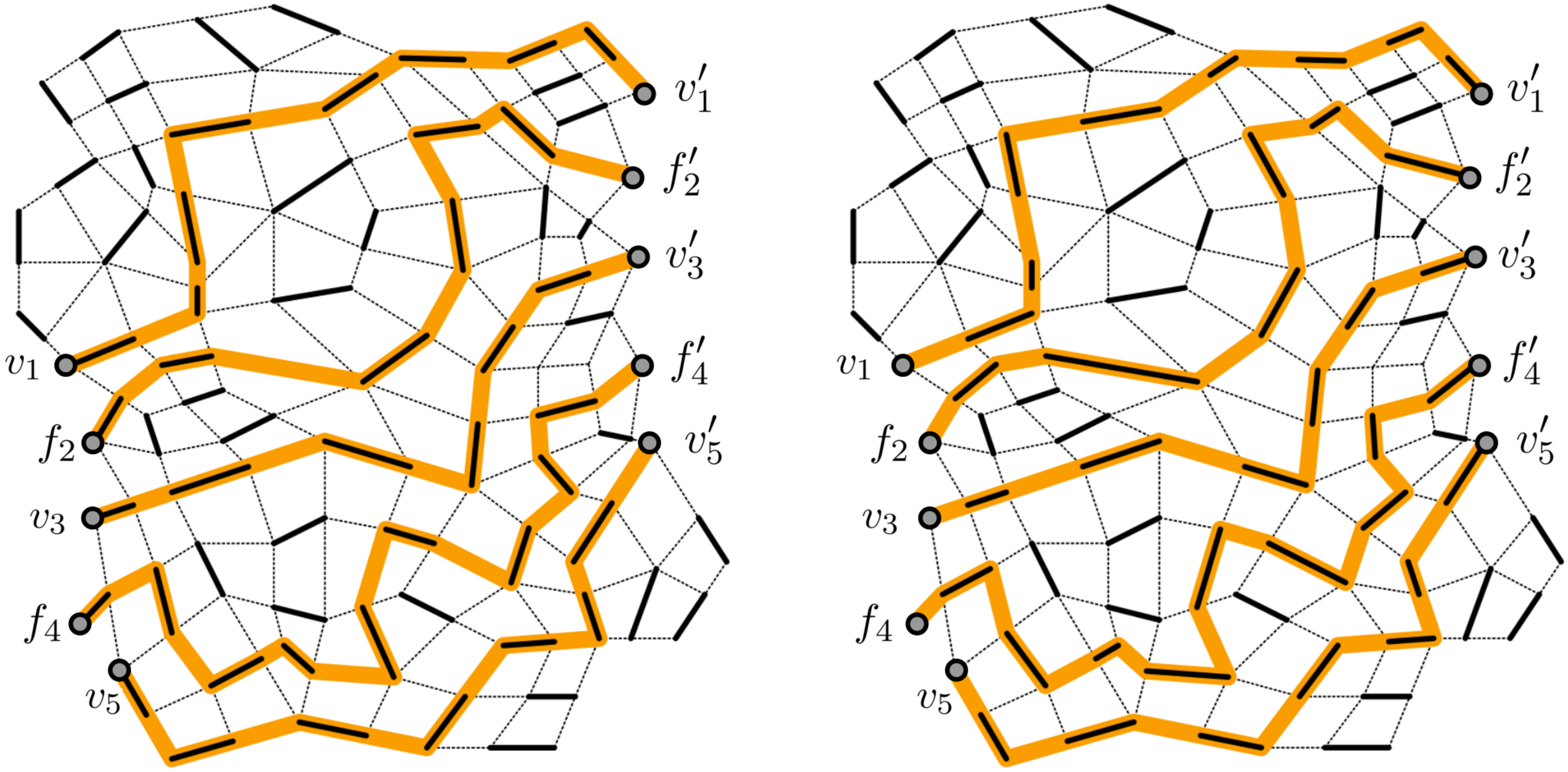}
    \caption{{\it Left.} An example of a perfect matching of $\widehat{H}_{G}\setminus \{v'_1,f'_2,v'_3,f'_4,v'_5\}$ that contains the unique perfect matching of the path $P_{i}\setminus \en_{P_i}$, $i=1,\dotsc,5$. {\it Right.} The corresponding perfect matching of $\widehat{H}_{G}\setminus \{v_1,f_2,v_3,f_4,v_5\}$ that contains the unique perfect matching of the path $P_{i}\setminus \st_{P_i}$, $i=1,\dotsc,5$.}
    \label{feb}
\end{figure}

\begin{proof}
Clearly, part (a) is the special case $I=\emptyset$ of part (b).
  
To prove part (b), let $\mu$ be a perfect matching of $\widehat{H}_{G}\setminus \{v'_1,f'_2,v'_3,f'_4,\dotsc,f'_{2n},v'_{2n+1}\}$ which contains all the edges of the unique perfect matching of the path $P_i\setminus \en_{P_i}$, for $i\in I$ (see the picture on the left in Figure \ref{feb} for an example). For an arbitrary $i\in[n+1]$, consider the vertex $v_{2i-1}$. Start from it and follow the edge of $\mu$ that contains it, and ``glide'' (as described in Section 3) along the frame until the next original vertex; from there, follow again the incident edge of $\mu$, and keep gliding this way until we get stuck. Clearly, we can only get stuck at one of the vertices $v'_{2j-1}$. This gives rise to $n+1$ paths on the frame; denote by $Q_{2i-1}$ the path starting at $v_{2i-1}$.

Note that these paths are non-intersecting. Indeed, if two paths $Q_{2i-1}$ and $Q_{2j-1}$ had a common point, by adding to these a portion of the boundary of the infinite face, a cycle in $G$ would be created, and Lemma \ref{txx} would lead to a contradiction. Therefore, $Q_1,Q_3,\dotsc,Q_{2n+1}$ are non-intersecting paths on the frame; then condition $(iii)$ implies that they connect $v_1,v_3,\dotsc,v_{2n+1}$ to $v'_1,v'_3,\dotsc,v'_{2n+1}$, respectively.

Now consider an arbitrary $i\in[n]$, and construct the path $Q_{2i}$ by gliding from $f_{2i}$. This time, the gliding takes place on the dual frame. Note that the path $Q_{2i}$ is confined to stay strictly between the previously constructed paths $Q_{2i-1}$ and $Q_{2i+1}$ (indeed, it cannot touch either of them, as that would lead to a vertex incident to two edges of $\mu$). Since these two paths end at $v'_{2i-1}$ and $v'_{2i+1}$, which are next-nearest neighbors in $G$ (by property $(ii)$ at the beginning of this section), it follows that $Q_{2i}$ must end at $f'_{2i}$. Thus, $Q_1,Q_2,\dotsc,Q_{2n+1}$ are non-intersecting paths in $\widehat{H}_G$ connecting $v_1,f_2,v_3,\dotsc,f_{2n},v_{2n+1}$ to  $v'_1,f'_2,v'_3,\dotsc,f'_{2n},v'_{2n+1}$, respectively. Furthermore, note that our construction implies that for $i\in I$, the path $Q_i$ must coincide with the path $P_i$ in the statement of the theorem (Figure \ref{feb} illustrates a situation when $I=[2n+1]$, so $Q_i=P_i$ for all $i$).

The required bijection can now easily be constructed. Let $\widetilde{\mu}$ be obtained from $\mu$ by shifting along the paths $Q_1,Q_2,\dotsc,Q_{2n+1}$, i.e.\ by including all the edges of $\mu$ off these paths, and all the edges on these paths which are not in $\mu$. We claim that $\mu\mapsto\widetilde{\mu}$ is the required bijection.

Indeed, by construction, $\widetilde{\mu}$ is a perfect matching of $\widehat{H}_{G}\setminus \{v_1,f_2,v_3,f_4,\dotsc,f_{2n},v_{2n+1}\}$, and it contains the unique perfect matching of $P_i\setminus\st_{P_i}$ for all $i\in I$ (since $Q_i=P_i$ for such $i$). So our map is well-defined. Due to property $(i)$ at the beginning of this section, this map can be inverted, so it is a bijection. It is also weight-preserving, as the two edges of $H_G$ contained in any edge of $G$ have the same weight, and the same holds for the two edges of $H_G$ contained in any edge of the planar dual $G^*$.
\end{proof}
  
As an application of Theorem \ref{tea}, we answer a question posed by  Corteel, Huang and Krattenthaler in \cite{Corteeletal2023AztecT1}, asking to find a bijection between the sets of perfect matchings of two families of graphs on the square grid, called generalized Aztec triangles.

The Aztec triangles were introduced by Di Francesco and Guitter in \cite{dfg}.
They arose naturally as regions on the square lattice whose domino tilings can be identified with the sets of configurations of the twenty-vertex model of statistical mechanics having a certain type of domain wall boundary conditions.


Start with a lattice square $S_{2n}$ of side-length $2n$ on $\Z^2$, and cut it in two congruent parts by a zig-zag lattice path $P$ which leaves the diagonal unit squares of $S_{2n}$ alternately below and above~$P$. Place above the resulting top half of $S_{2n}$ (consisting of the bottom 10 rows in the top left picture in Figure \ref{fec}, which illustrates the case $n=5$) the top half of an Aztec diamond of order $(n-1)$, in a right-justified manner. The resulting region ${\mathcal T}_n$ is the Aztec triangle of order $n$ (for $n=5$, this is shown on the top left in Figure \ref{fec}).


In \cite{DiFrancesco202120Vmodel} Di Francesco conjectured that the number of domino tilings of the Aztec triangle of order $n$ is
%
\begin{equation}
    2^{\frac{n(n-1)}{2}}\prod_{i=0}^{n-1}\frac{(4i+2)!}{(n+2i+1)!},
\label{eea}
\end{equation}
a formula reminiscent of the notorious expression $\prod_{i=0}^{n-1}\frac{(3i+1)!}{(n+i)!}$ giving the number of alternating sign matrices of order $n$.

\begin{figure}
    \centering
    \includegraphics[width=.5\textwidth]{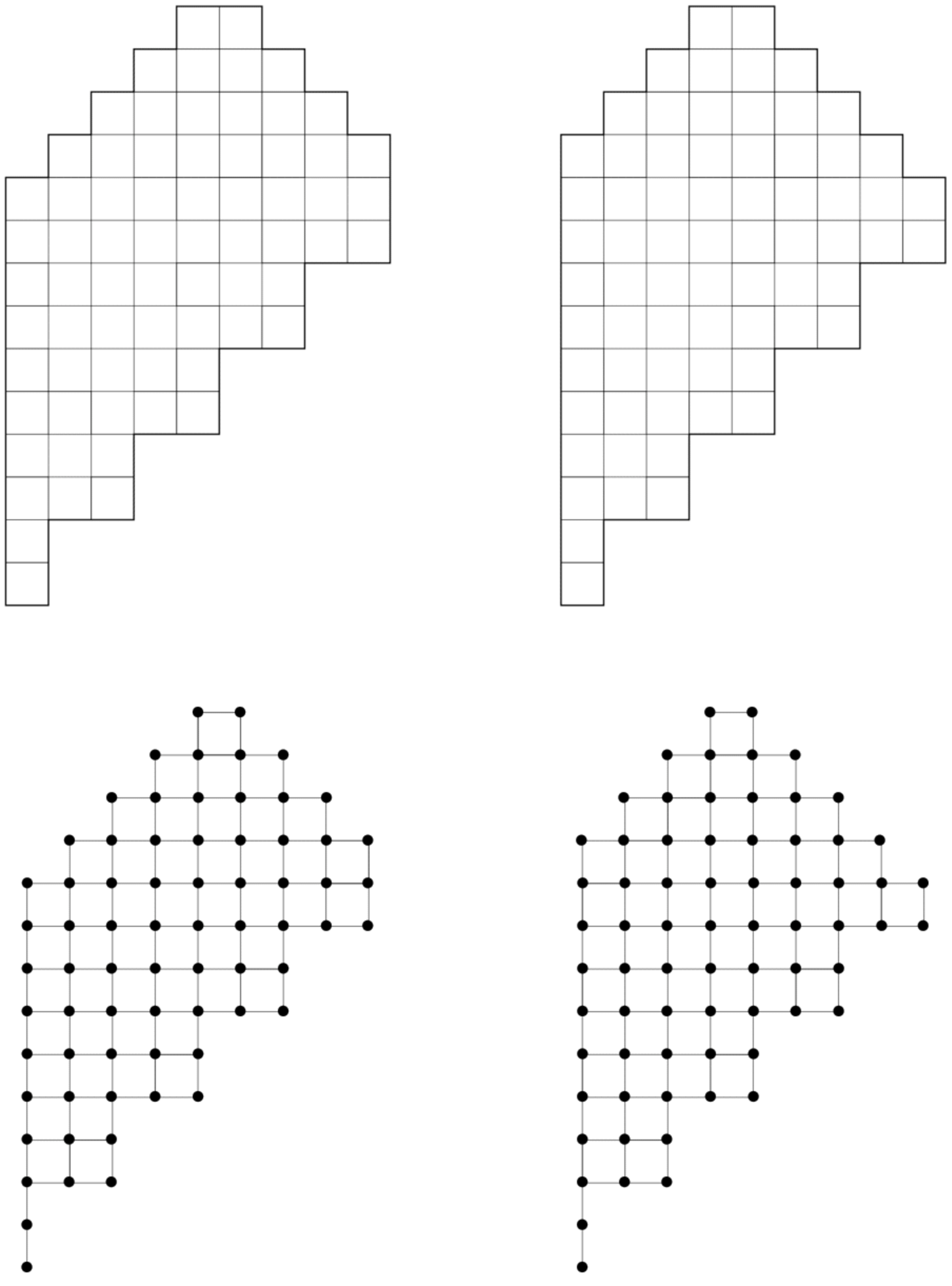}
    \caption{The Aztec triangles ${\mathcal T}_5$ (top left) and ${\mathcal T}'_5$ (top right). The planar dual graphs of these regions are shown below them. Our bijection is between perfect matchings of these planar dual graphs.}
    \label{fec}
\end{figure}

This conjecture of Di Francesco was recently proved in \cite{Corteeletal2023AztecT1} and \cite{Koutschanetal2024AztecT2} (see also \cite{df} for a generalization in a different direction). In \cite{Corteeletal2023AztecT1}, more general regions are considered, and simple product formulas are proved for them. One of these results concerns a variation of the Aztec triangle region~${\mathcal T}_n$ defined above, obtained using the same construction, with the one difference that the half Aztec diamond is placed above the half square in a {\it left}-justified manner (for $n=5$, this is illustrated on the top right in Figure \ref{fec}); denote this region by ${\mathcal T}'_n$. The formula proved in \cite{Corteeletal2023AztecT1} for the number of domino tilings of ${\mathcal T}'_n$ is identical with expression \eqref{eea}, which shows that ${\mathcal T}_n$ and ${\mathcal T}'_n$ have the same number of domino tilings. After pointing this out, the authors of \cite{Corteeletal2023AztecT1} write: ``We leave it as an open problem to find a bijection between these two families of domino tilings.''

We provide such a bijection, using Theorem \ref{tea}.

\begin{cor}
    There is a natural bijection between the sets of domino tilings of ${\mathcal T}_n$ and ${\mathcal T}'_n$.
\label{teb}
\end{cor}

\begin{figure}
    \centering
    \includegraphics[width=.8\textwidth]{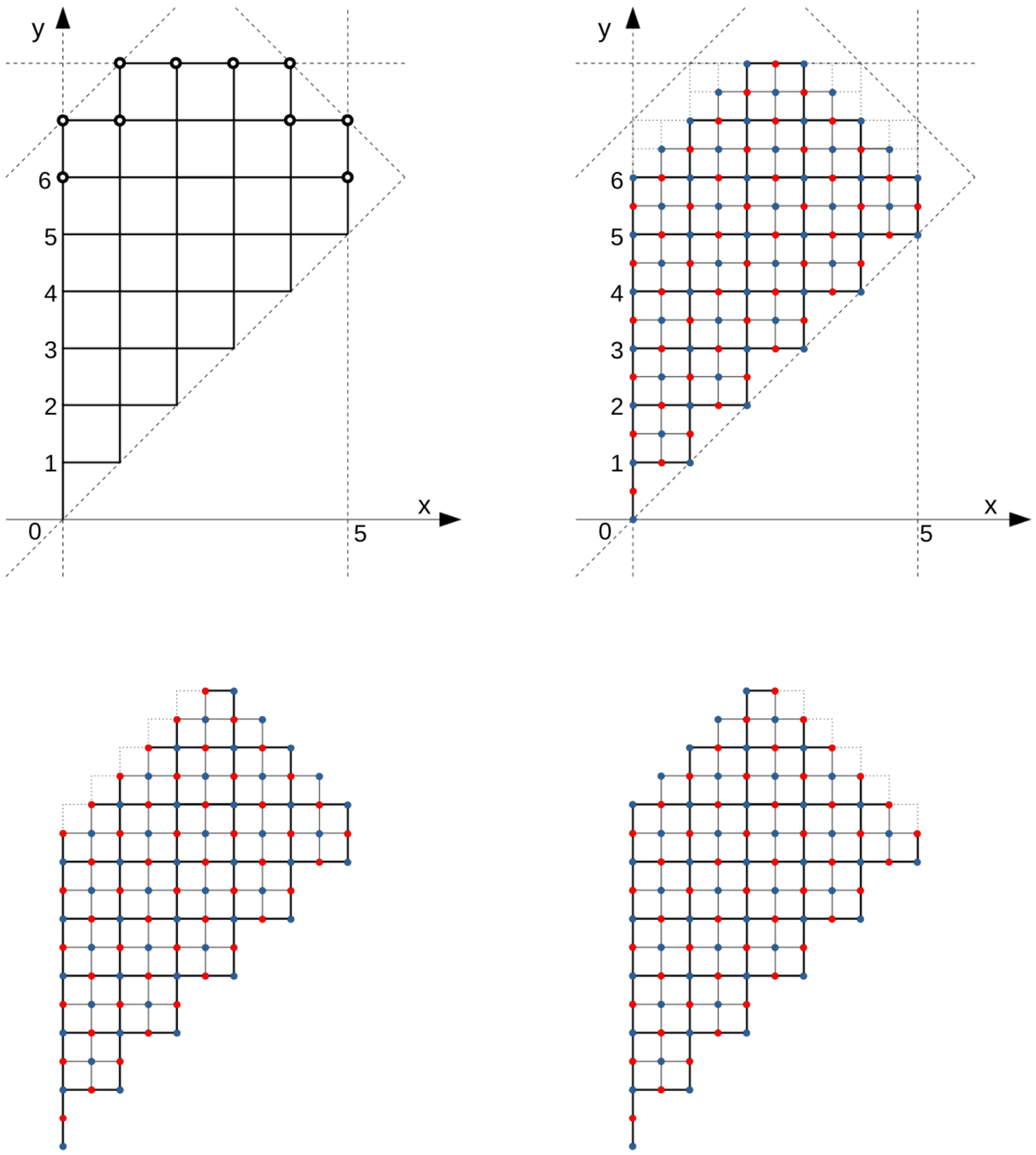}
    \caption{Illustrating the proof of Corollary \ref{teb} for $n=2m$ (here $m=3$). {\it Top left.} The graph $G$ with $v_{1},\ldots,v_{2m-1}$ (the five vertices on the top left) and $v'_{1},\ldots,v'_{2m-1}$ (the five vertices on the top right) marked. {\it Top right.} The corresponding graph $\widehat{H}_{G}$. {\it Bottom left.} $\widehat{H}_G\setminus \{v_1,f_2,v_3,\dotsc,v_{2m-1}\}$ is isomorphic to ${\mathcal T}_n$. {\it Bottom right.}  $\widehat{H}_G\setminus \{v'_1,f'_2,v'_3,\dotsc,v'_{2m-1}\}$ is isomorphic to ${\mathcal T}'_n$.
    }
    \label{fed}
\end{figure}

\begin{proof}
  The proof is slightly different for even and odd $n$. Let $n=2m$. On the $xy-$coordinate plane, consider the closed region enclosed by the six lines
  $x=0$, $x=2m-1$ $y=x$, $y=x+2m+1$, $y=-x+4m$, and $y=3m-1$. Let $G$ be the induced subgraph of $\mathbb{Z}^{2}$ whose vertices are the integer lattice points in this region, with all edges weighted by $1$; for $m=3$, the graph $G$ is shown on the top left in Figure \ref{fed}. The desired bijection will follow by applying Theorem \ref{tea} to the graph $G$, with the $v_i$'s and $v'_i$'s chosen as described below.

Choose $v_{2i-1}=(m-i,3m-i), v'_{2i-1}=(m-1+i,3m-i)$, for $i\in[m]$, and $v_{2j}=(m-1-j,3m-j), v'_{2j}=(m+j,3m-j)$, for\footnote{When $m=1$, $[m-1]=[0]\coloneqq\emptyset$.} $j\in[m-1]$, and consider the dual refinement $H_{G}$ of $G$ --- regarded as a subgraph of $(\frac12\Z)^2$. The graph $\widehat{H}_G$ is obtained by smashing in vertices $v_2,v_4,\dotsc,v_{2m}$ and  $v'_2,v'_4,\dotsc,v'_{2m}$; the corresponding smashed in vertices are $f_{2j}=(m-\frac{1}{2}-j,3m-\frac{1}{2}-j)$ and $f'_{2j}=(m-\frac{1}{2}+j,3m-\frac{1}{2}-j)$, $j=1,\dotsc,n$.

One readily checks that the graph $\widehat{H}_G\setminus \{v_1,f_2,v_3,\dotsc,v_{2m-1}\}$ is precisely the planar dual of the Aztec triangle region ${\mathcal T}_{2m}$. Moreover, the graph $\widehat{H}_G\setminus \{v'_1,f'_2,v'_3,\dotsc,v'_{2m-1}\}$ is the planar dual of ${\mathcal T}'_{2m}$. Thus the statement follows directly by Theorem \ref{tea}(a).



\begin{figure}
    \centering
    \includegraphics[width=.9\textwidth]{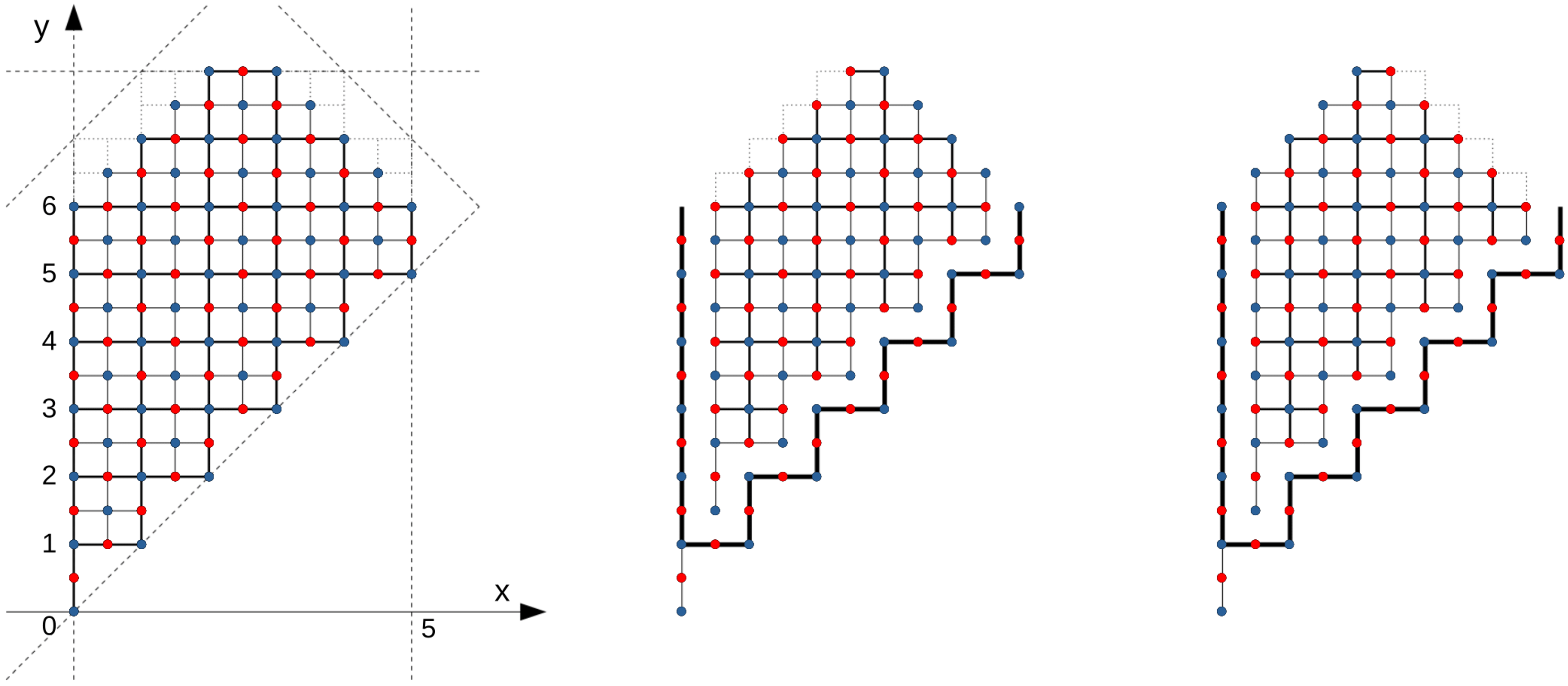}
    \caption{Illustrating the proof of Corollary \ref{teb} for $n=2m-1$ (here $m=3$). {\it Left.} The graph $\widehat{H}_{G}$ is exactly the same graph as in the case $n=2m$. {\it Center and right.} The vertices $v_{1},f_{2},v_{3},f_{4},v_{5},v'_{1},f'_{2},v'_{3},f'_{4},v'_{5}$ are also the same as for $n=2m$. The path $P_{2m-1}$ joining $v_{{2m-1}}$ and $v'_{2m-1}$ is shown in bold lines. One readily sees that the subgraphs above these paths are exactly the planar dual graphs of ${\mathcal T}_{2m-1}$ and ${\mathcal T}'_{2m-1}$ (compare with the bottom pictures in Figure \ref{fec}).}
    \label{fee}
\end{figure}

Suppose now that $n=2m-1$. We use the same graph $G$ and the same vertices $v_1,f_2,v_3,\dotsc,v_{2m-1}$ and $v'_1,f'_2,v'_3,\dotsc,v'_{2m-1}$
as in the case $n=2m$. Define $P_{2m-1}$ to be the path joining vertices $v_{2m-1}=(0,2m),(0,1), (1,1), (1,2), (2,2),\ldots,(2m-1,2m-1)$, and $(2m-1,2m)=v'_{2m-1}$ (this path is illustrated in Figure \ref{fee}). The statement will follow by applying Theorem \ref{tea}(b) with $n=m-1$ and $I=\{2m-1\}$. 

Indeed, this yields a bijection between the set of perfect matchings of $\widehat{H}_{G}\setminus \{v_1,f_2,v_3,\dotsc,v_{2m-1}\}$ that contain all edges of the unique perfect matching of the path $P_{2m-1}\setminus v_{2m-1}$, and the set of perfect matchings of $\widehat{H}_{G}\setminus \{v'_1,f'_2,v'_3,\dotsc,v'_{2m-1}\}$ that contain all edges of the unique perfect matching of the path $P_{2m-1}\setminus v'_{2m-1}$. However, due to forced edges along the path $P_{2m-1}$ and at the two vertices below it, these perfect matchings can be identified with the perfect matchings of ${\mathcal T}_{2m-1}$ and those of ${\mathcal T}'_{2m-1}$, respectively. This completes the proof.  
\end{proof}

\medskip
The second main result of this section is an extension of Temperley's theorem (see Theorem~\ref{tda}). It involves the following special type of spanning forests of a planar graph. Let $G$ be a plane graph. A spanning forest $F$ of $G$ is called {\it banded} if there exist vertices $u_1,\dotsc,u_k,u'_k,\dotsc,u'_1$ --- called the {\it distinguished points of $F$} ---, in counterclockwise order on the boundary of the infinite face of $G$, so that $F$ consists of $k$ different trees, each containing one pair of vertices $(u_i,u'_i)$, $i=1,\dotsc,k$. Connected components of a banded spanning forest are sometimes referred to as {\it bands}.

Given a banded spanning forest $F$ of a plane graph $G$, consider the subgraph $F^*$ of the planar dual graph $G^*$ obtained by including all the vertices of $G^*$, and all the edges of $G^*$ for which the corresponding edge of $G$ is not contained in $F$. One readily sees that, since $F$ is banded, $F^*$ does not contain any cycle (such a cycle would surround at least one original vertex $v$, and the connected component of $F$ containing $v$ could not include any boundary vertex of $G$). Therefore, $F^*$ is a spanning forest of $G^*$; we call it the {\it dual spanning forest of $F$}.

\begin{lemma}
Let $F$ be a banded spanning forest whose distinguished points are $u_1,\dotsc,u_k,u'_k,\dotsc,u'_1$, and let $T^*$ be a connected component of the dual spanning forest $F^*$.

 
$(${\rm a}$)$. There exists a vertex in $T^*$ adjacent to $z$ in $G^*\cup z$, the dual graph\footnote{ Recall that the dual graph of $G$ consists of the planar dual $G^*$ with an additional vertex $z$ corresponding to the infinite face of $G$, so that $z$ is connected by an edge to all the vertices of $G^*$ corresponding to faces of $G$ that share an edge with the infinite face.} of~$G$.

$(${\rm b}$)$. $T^*$ has either exactly one or exactly two such vertices. If there are two such vertices $f$ and $f'$, then $T^*$ is contained between two consecutive bands of $F$. If these are the $(i-1)$st and $i$th bands, then one of the two edges $\{f,z\}$, $\{f',z\}$ of the dual graph $G^*\cup z$ crosses the counterclockwise arc $(u_{i-1},u_{i})$, and the other the counterclockwise arc $(u'_{i},u'_{i-1})$, of the boundary of the infinite face of~$G$.

\label{tecc}
\end{lemma}  

If, in the context of Lemma \ref{tecc}, $T^*$ contains two vertices adjacent to $z$, we call $T^*$ a {\it channel} of $F$; if it contains just one such vertex, we call it a  a {\it bay}  of $F$\footnote{ Thus, the bands of $F$ could be thought of as islands.}.

\begin{proof} (a). Suppose there is no such vertex in $T^*$. Then it is possible to walk around $T^*$ in $G$, along a walk that separates it from the other connected components of $F^*$, without using any edges on the infinite face of $G$. But each edge of $G$ in such a walk must be in $F$, as its dual is not in $F^*$ (since this dual edge is part of our walk, which separates  $T^*$ from the other connected components of $F^*$). This produces a cycle contained in $F$, a contradiction.

(b). Partition the boundary of the infinite face of $G$ into the counterclockwise arcs
\begin{equation}
(u_1,u_2), (u_2,u_3),\dotsc,(u_{k-1},u_k),(u_k,u'_k),(u'_k,u'_{k-1}),\dotsc,(u'_2,u'_1),(u'_1,u_1).
\label{eeb}
\end{equation}
Suppose $T^*$ contains two distinct vertices $f$ and $f'$ adjacent to $z$ in $G^*\cup z$, and that both edges $\{f,z\}$ and $\{f',z\}$ cross the boundary of the infinite face of $G$ along the same arc in the list \eqref{eeb}.

Let $P^*$ be the path in $T^*$ connecting $f$ and $f'$. Then $P^*$, together with the edges $\{f,z\}$ and $\{f',z\}$, forms a cycle $C^*$ of the dual graph $G^*\cup z$. The cycle $C^*$ contains at least one vertex $v$ of $G$ in its interior. Then the connected component of $F$ containing $v$, being confined to the interior of $C^*$, cannot contain any of the vertices $u_1,\dotsc,u_n,u'_1,\dotsc,u'_n$, a contradiction.

Therefore, the edges $\{f,z\}$ and $\{f',z\}$ cross the boundary of the infinite face of $G$ along distinct arcs of the list \eqref{eeb}.

\begin{figure}
    \centering
    \includegraphics[width=.85\textwidth]{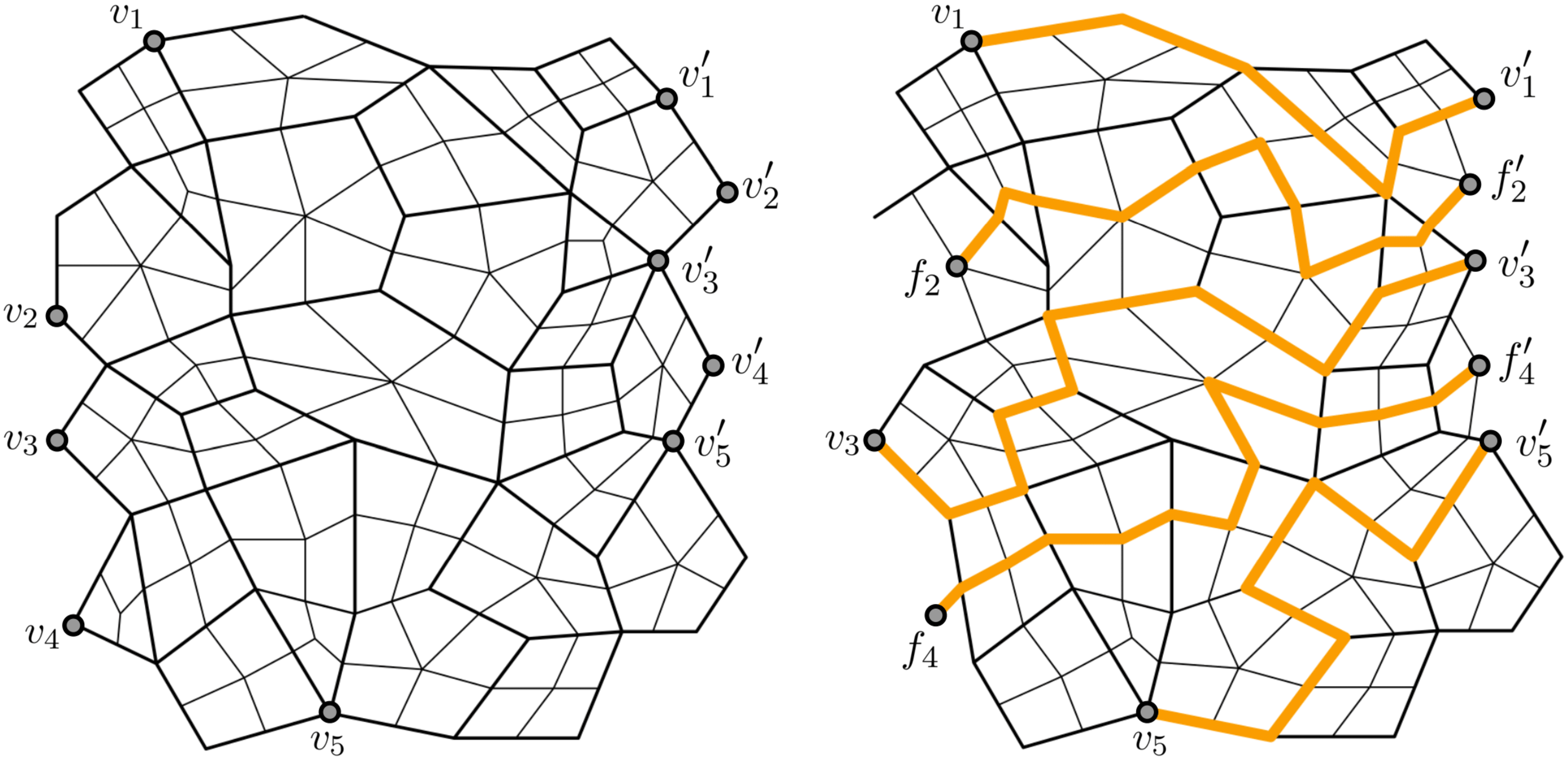}
    \caption{{\it Left.} An instance of the graph $H_{G}$; bold lines in the picture indicate the graph $G$ (the vertices $v_1,\ldots,v_5,v'_5,\ldots,v'_1$ are marked and labeled). {\it Right.} The graph $\widehat{H}_{G}$ --- obtained from the left picture by smashing in the four vertices $v_2,v_4,v'_4,v'_2$ --- together with five non-intersecting paths $P_1,P_2,P_3,P_4,P_5$ (from top to bottom).}
    \label{fef}
\end{figure}

Let the bands of $F$ be $B_1,\dotsc,B_k$. Since $B_i$ contains $u_i$ and $u'_i$, it contains also a path $P_i$ connecting them. The paths $P_1,\dotsc,P_k$ divide the interior of the boundary of the infinite face of $G$ into $k-1$ strips, plus a top cap (the portion above $P_1$) and a bottom cap (the portion below $P_k$). It follows that $T^*$ must be contained in one of the $k-1$ strips between the bands of $F$, or in one of the top or bottom caps.

Suppose $T^*$ is contained in one of the caps, without loss of generality the top cap. Then for any two distinct vertices $f$ and $f'$ adjacent to $z$ in $G^*\cup z$, both edges $\{f,z\}$ and $\{f',z\}$ must cross the boundary of the infinite face of $G$ along the arc $(u'_1,u_1)$. Therefore, the above arguments implies that $T^*$ contains a unique vertex adjacent to $z$ in $G^*\cup z$.

For the remaining case, consider the situation when $T^*$ is contained in the strip between the bands $B_{i-1}$ and $B_i$. Assume that $T^*$ has more than one  vertex adjacent to $z$ in $G^*\cup z$, and let $f$ and $f'$ be two of them. Then the above argument shows that one of the edges $\{f,z\}$ and $\{f',z\}$ must cross the boundary of the infinite face of $G$ along the arc $(u_{i-1},u_i)$, while the other must cross it along $(u'_{i},u'_{i-1})$. It also follows that there can be no other such vertex of $T^*$ besides $f$ and $f'$, as then at least two crossings would occur along the same arc. Therefore, if $T^*$ has at least two such vertices, it must have exactly two. By part (a), the other option is for $T^*$ to have exactly one such vertex. This completes the proof.
\end{proof}

A {\it rooted spanning forest} is a spanning forest $F$ with a distinguished vertex (called {\it root}) in each connected component; in particular, a rooted spanning forest which is also banded has one root in each band. Each edge $e$ of a rooted spanning tree $F$ is oriented so that in the unique path $P$ connecting its midpoint to the root in the connected component containing $e$, the edge $e$ points toward this root.

\begin{figure}
    \centering
    \includegraphics[width=.85\textwidth]{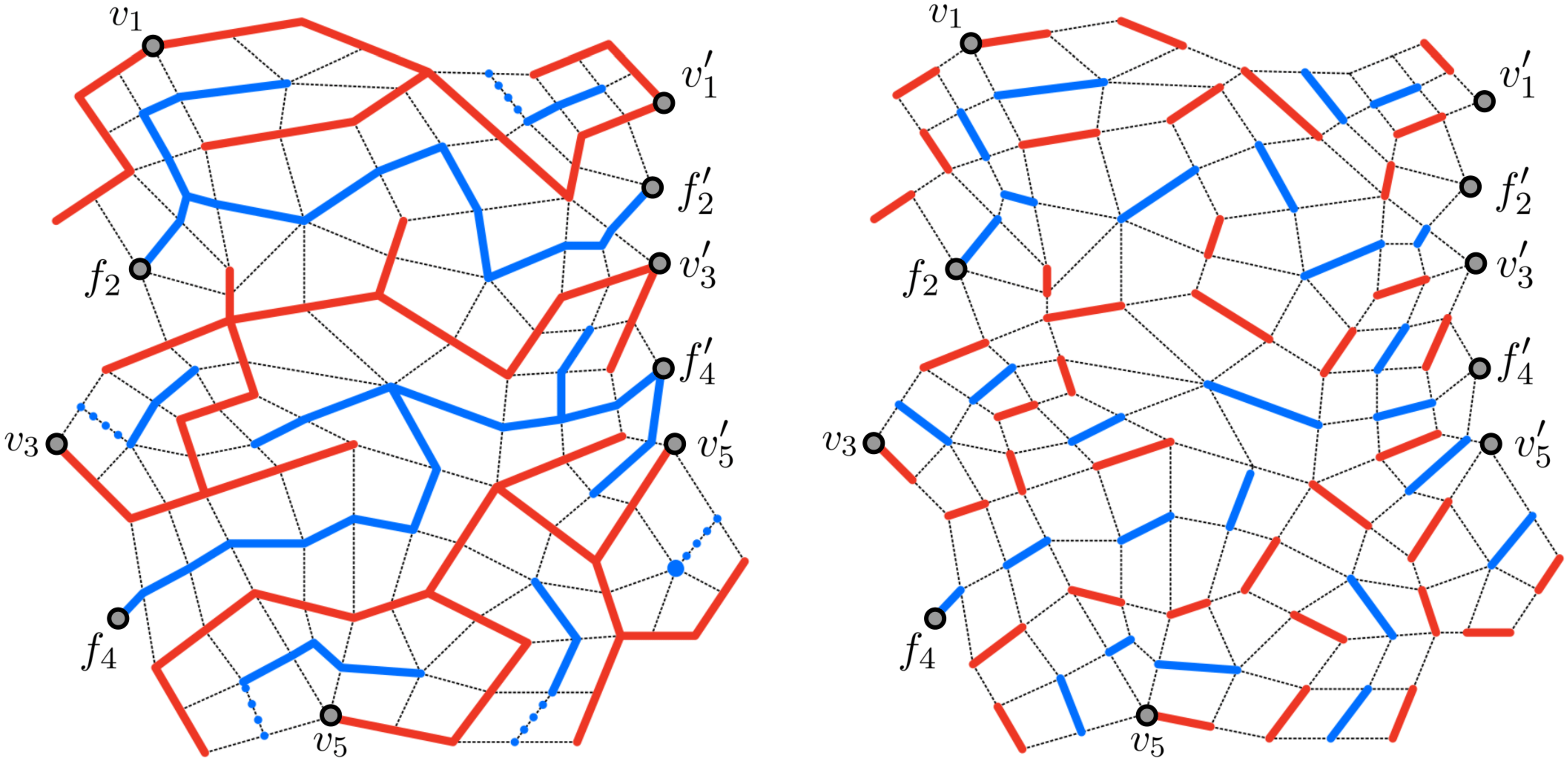}
    \caption{{\it Left.} An example of a spanning forest $F$ of $G\setminus \{v_2,v_4,\dotsc,v_{2n},v'_2,v'_4,\dotsc,v'_{2n}\}$ (shown in red) and the corresponding dual spanning forest $F^{*}$ (shown in blue solid lines), for the graph depicted on the left in Figure~\ref{fef}. {\it Right.} The corresponding perfect matching of $\widehat{H}_{G}\setminus \{v'_1,f'_2,v'_3,\dotsc,v'_{2n+1}\}$. One can obtain the left picture from the right picture by ``gliding,'' and the right picture from the left picture by selecting the tail half-edges of the rooted trees in $F$ and $F^*$ (each bay of $F$ is augmented by adding the dotted edge connecting it to $z$).}
    \label{feg}
\end{figure}

\begin{thm}
\label{tec}
Let $G$ be a weighted plane graph,
and weight all edges of its its planar dual graph $G^{*}$ by $1$\footnote{ The statement readily extends for an arbitrary weight on the edges of $G^*$ (see Remark 2). We present the result with weights 1 on $G^*$ because it essentially captures all generality, and is simpler to state (otherwise we need to define the weight of the spanning forest $F$ as the product of the weights of the edges of $F$ times the product of the weights of the edges in the channels.
.}. Weight each edge of the dual refinement $H_G$ by the weight of the edge of $G$ or $G^*$ that contains it\footnote{ As for Theorem \ref{tea}, weight by 1 each edge of $H_G$ connecting an edge-vertex on the infinite face to a face-vertex.}. Suppose the vertices $v_1,v_2,\dotsc,v_{2n+1},v'_{2n+1},\dotsc,v'_1$ satisfy the conditions $(ii)$--$(iv)$ stated at the beginning of this section $($see the picture on the left in Figure $\ref{fef}$ for an illustration; the underlying graph is the same as in Figure $\ref{fea}$, but the $v_i$'s are not required to be adjacent$)$. 

\medskip
$(${\rm a}$)$.\ There is a natural weight-preserving bijection between
%
the set of perfect matchings of
$\widehat{H}_{G}\setminus \{v'_1,f'_2,v'_3,\dotsc,v'_{2n+1}\}$
%
and
%
the set of banded spanning forests $F$ of
$G\setminus \{v_2,v_4,\dotsc,v_{2n},v'_2,v'_4,\dotsc,v'_{2n}\}$
rooted at $v'_1,v'_3,\dotsc,v'_{2n+1}$, with $v_{2i-1}$ and $v'_{2i-1}$ in the same band, and $f_{2i}$ and $f'_{2i}$ in the same channel, for all $i$.

\medskip
$(${\rm b}$)$. More generally, let $P_1,P_2,\dotsc,P_{2n+1}$ be non-intersecting paths on $H_G$ as in the statement of Theorem $\ref{tea}${\rm (b)}\footnote{ For ease of reference, we recall the conditions on these paths here: $P_{2i-1}$ is a path in $G$ connecting $v_{2i-1}$ to $v'_{2i-1}$, for $i=1,\dotsc,n+1$, while $P_{2i}$ is a path in the planar dual $G^*$ connecting $f_{2i}$ to $f'_{2i}$, for $i=1,\dotsc,n$. Recall also that for a path $P$, we denote by $\st_P$ the first $($starting$)$ vertex in the path, and by $\en_P$ the last $($ending$)$ vertex in the path; the paths $P_1,P_2,\dotsc,P_{2n+1}$ are regarded as starting at $v_1,f_2,v_3,f_4,\dotsc,f_{2n},v_{2n+1}$, respectively.}. 

Then for any subset $I\subseteq[2n+1]$, there is a weight-preserving bijection between

\medskip
$(1)$ the set of perfect matchings of
$\widehat{H}_{G}\setminus \{v'_1,f'_2,v'_3,\dotsc,v'_{2n+1}\}$
that contains all edges of the $($unique$)$ perfect matching of the path $P_{i}\setminus \en_{P_i}$ for $i\in I$,

\medskip
and

\medskip
$(2)$ the set of banded spanning forests $F$ of $G\setminus \{v_2,v_4,\dotsc,v_{2n},v'_2,v'_4,\dotsc,v'_{2n}\}$ rooted at the vertices $v'_1,v'_3,\dotsc,v'_{2n+1}$
such that

\medskip
$\bullet$ $v_{2i-1}$ and $v'_{2i-1}$ are in the same band, and  $f_{2i}$ and $f'_{2i}$ are in the same channel, for all $i$

$\bullet$ for all odd $i\in I$, the path $P_i$ is contained in $F$

$\bullet$ for all even $i\in I$, the path $P_i$ is contained in a channel of $F$.

\end{thm}

\begin{proof}
We prove part (a) first. Let $\mu$ be a perfect matching of $\widehat{H}_{G}\setminus \{v'_1,f'_2,v'_3,\dotsc,v'_{2n+1}\}$. Based on it, we construct a spanning forest $F_\mu$ of $G_0:=G\setminus \{v_2,v_4,\dotsc,v_{2n},v'_2,v'_4,\dotsc,v'_{2n}\}$ as follows.

``Glide'' from each vertex $u$ of $G_0$, using the construction described in Section 3. More precisely, starting at $u$, follow the edge $e$ of $\mu$ containing it\footnote{ Such an edge always exist, as the vertex set of $G_0$ is contained in the vertex set of the graph $\widehat{H}_{G}\setminus \{v'_1,f'_2,v'_3,\dotsc,v'_{2n+1}\}$, of which $\mu$ is a perfect matching.}, and continue along the edge $\overline{e}$ of $G$ containing $e$ until reaching the other end of $\overline{e}$; from here, follow a new edge of $\mu$ and move along it
to the next encountered vertex of $G_0$,
and repeat this until we cannot glide anymore from the vertex of $G$ that we reached. Clearly, this happens only if the vertex we ended at is one of $v'_1,v'_3,\dotsc,v'_{2n+1}$. Thus we obtain a path\footnote{ There can be no self-intersection in $Q_u$, as that would produce a cycle, and then Lemma \ref{txx} would lead to a contradiction.} $Q_u$ on the edges of $G_0$, starting at $u$ and ending at one of $v'_1,v'_3,\dotsc,v'_{2n+1}$. Orient the edges of $Q_u$ so that they point towards the ending point; note that under this orientation, all the edges of $\mu$ contained in $Q_u$ are tail half-edges. Define $F_\mu$ to be the union of all these paths $Q_u$, over all vertices $u$ of $G_0$. We claim that $F_\mu$ is a banded spanning forest of $G_0$ rooted at the vertices $v'_1,v'_3,\dotsc,v'_{2n+1}$, with $v_{2i-1}$ and $v'_{2i-1}$ in the same band, and $f_{2i}$ and $f'_{2i}$ in the same channel, for all $i$.

Note that, by construction, if two such gliding paths $P_u$ and $P_v$ meet at a vertex, they will coincide from that point onward. This implies that if $V_{2i-1}$ consisting of those vertices $u$ of $G_0$ for which $P_u$ ends at $v'_{2i-1}$, then $T_{2i-1}:=\cup_{u\in V_{2i-1}} Q_u$ is a tree rooted at $v'_{2i-1}$. Since these trees are determined by fixing the ending point of the gliding paths, they are disjoint. Therefore, since each vertex of $G_0$ is the starting point of a gliding path, the union of these trees --- which is $F$ --- is a spanning forest of $G_0$, consisting of the trees $T_1,T_3,\dotsc,T_{2n+1}$, rooted at the vertices $v'_1,v'_3,\dotsc,v'_{2n+1}$, respectively.

\begin{figure}
    \centering
    \includegraphics[width=.85\textwidth]{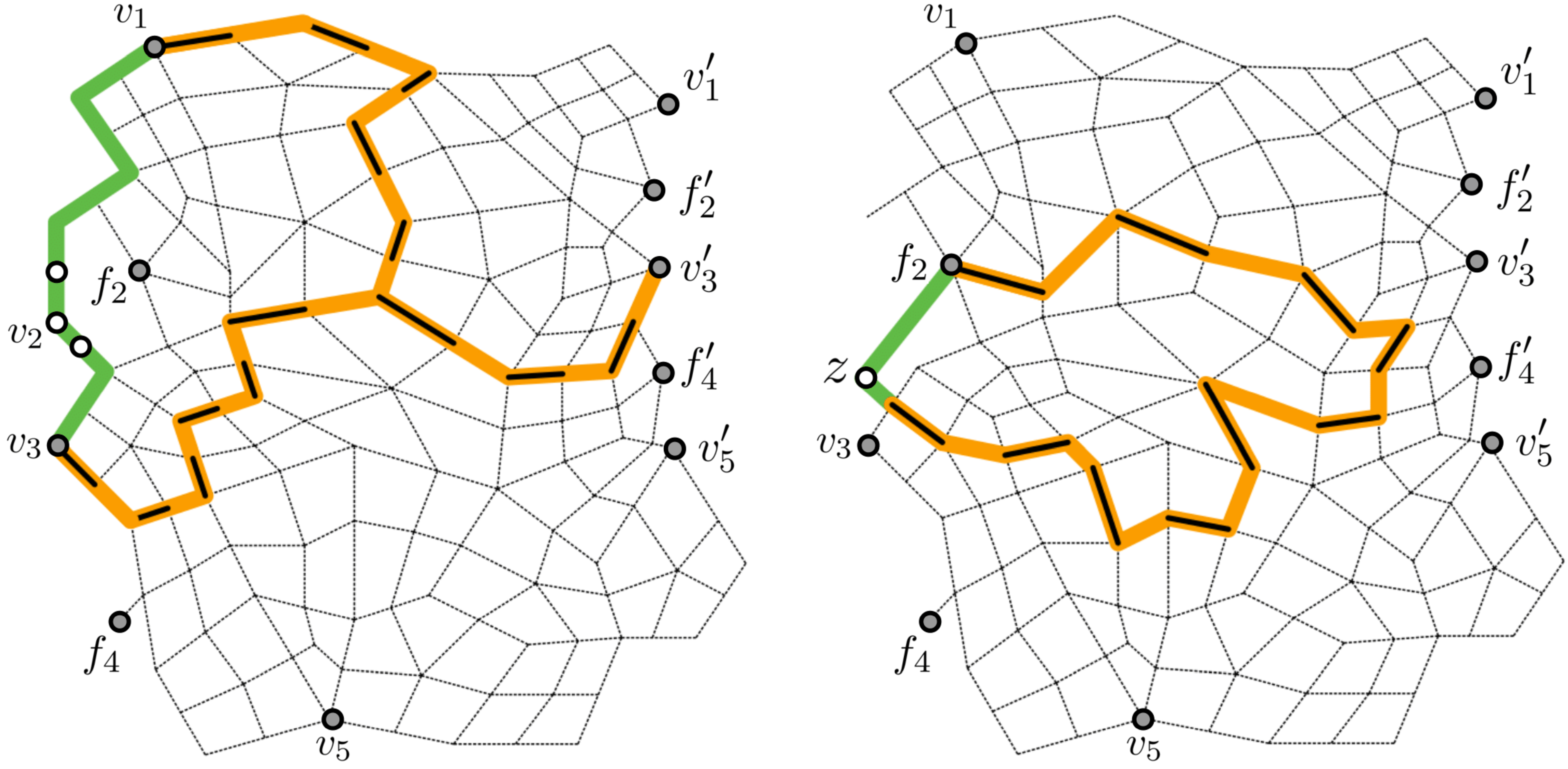}
    \caption{The left picture illustrates why two paths starting from $v_{1}$ and $v_{3}$ that are generated by sequences of gliding moves do not intersect. The vertex $v_{2}$ and its two adjacent edge-vertices are denoted by empty circles. The picture on the right illustrates why the gliding path starting from $f_{2}$ cannot terminate at an edge-vertex on the boundary in between $v_{1}$ and $v_{3}$. The face-vertex for the infinite face, denoted by $z$, is marked by an empty circle.}
    \label{fei}
\end{figure}

To see that $F$ is banded, note that the paths\footnote{ For notational simplicity, we write $Q_{2i-1}$ for the path $Q_{v_{2i-1}}$ obtained by gliding from $v_{2i-1}$.} $Q_1,Q_3,\dotsc,Q_{2n+1}$ are non-intersecting. 
Indeed, suppose $Q_{2i-1}$ and $Q_{2j-1}$ ($i<j$) do intersect, and let $u$ be their first common vertex (this is illustrated in the picture on the left in Figure \ref{fei}). Then if $P$ is the path in $G$ connecting the endpoints of the counterclockwise arc $(v_{2i-1},v_{2j-1})$, the union of $P$ with the initial portions of $Q_{2i-1}$ and $Q_{2j-1}$ up to $u$ forms a cycle $C$ in $G$. But then by Lemma~\ref{txx}, the number of vertices of $\widehat{H}_{G}\setminus \{v'_1,f'_2,v'_3,\dotsc,v'_{2n+1}\}$ in the interior of $C$ is odd, contradicting the fact that they are matched up in pairs by $\mu$. Therefore, using also condition $(iii)$, it follows that $Q_{2i-1}$ connects $v_{2i-1}$ to $v'_{2i-1}$, and thus these two vertices are in the same connected component $T_{2i-1}$ of $F$, for $i=1,\dotsc,n+1$. This shows that $F$ is banded, with $v_{2i-1}$ and $v'_{2i-1}$ in the same band for all $i$.

To see that $f_{2i}$ and $f'_{2i}$ are in the same channel of $F$, note that, since $f_{2i}$ is in the counterclockwise arc $(v_{2i-1},v_{2i+1})$ of the boundary of $\widehat{H}_{G}\setminus \{v'_1,f'_2,v'_3,\dotsc,v'_{2n+1}\}$, the gliding path $Q_{2i}$ starting at\footnote{ Recall that for face-vertices of $\widehat{H}_{G}\setminus \{v'_1,f'_2,v'_3,\dotsc,v'_{2n+1}\}$, gliding takes place on the ``dual frame'', i.e.\ along the edges of the dual graph $G^*\cup z$ of $G$.} $f_{2i}$ is confined to stay in the strip between the above defined paths $Q_{2i-1}$ and $Q_{2i+1}$ (indeed, a crossing point between a gliding path on the frame and one on the dual frame cannot exist, as such a vertex would be incident to two different edges of $\mu$). Furthermore, the gliding path $Q_{2i}$ cannot exit this strip along the counterclockwise arc $(v_{2i-1},v_{2i+1})$, because this would give rise to a cycle in the dual graph $G^*\cup z$ (see the picture on the right in Figure \ref{fei}), which by Lemma \ref{txx} contains an odd number of vertices of $\widehat{H}_{G}\setminus \{v'_1,f'_2,v'_3,\dotsc,v'_{2n+1}\}$ in its interior, contradicting the fact that they are matched by $\mu$.
Therefore, due to the fact that $v'_{2i-1}$ and $v'_{2i+1}$ are next-nearest neighbors on the boundary of the infinite face of $G$ (this is where condition $(ii)$ stated at the beginning of this section is essential!), the gliding path $Q_{2i}$ must end at $f'_{2i}$. This completes the proof of our claim, and shows that the above-described map $\mu\mapsto F_\mu$ is well-defined.

In order to describe the converse construction, we also need to glide from each face-vertex $f$ of $\widehat{H}_{G}\setminus \{v'_1,f'_2,v'_3,\dotsc,v'_{2n+1}\}$. Each such gliding path must end
either
at one of the vertices $f'_{2},f'_4,\dotsc,f'_{2n}$,
or at $z$. Orient the edges of these paths so that they point towards the ending point; again, note that under this orientation, all the edges of $\mu$ contained in these paths are tail half-edges. The same arguments as above show that the union of all the gliding paths starting at face-vertices --- with the convention that from the gliding paths ending at $z$ we discard the last edge --- forms a forest. One readily verifies that this forest is precisely the dual forest $F^*$ of $F$.
The connected component of $F^*$ which contains $f'_{2i}$ is a tree rooted at $f'_{2i}$, for all $i$. The remaining connected components of $F^*$ are trees with all edges oriented towards $z$.

We can now state the converse construction. Given a banded spanning forest of $G_0$ rooted at the vertices $v'_1,v'_3,\dotsc,v'_{2n+1}$, with $v_{2i-1}$ and $v'_{2i-1}$ in the same band, and $f_{2i}$ and $f'_{2i}$ in the same channel, for all $i$, we define a perfect matching $\mu_F$ of $\widehat{H}_{G}\setminus \{v'_1,f'_2,v'_3,\dotsc,v'_{2n+1}\}$ as follows. Consider the dual spanning forest $F^*$ of $F$. If a connected component $T^*$ of it contains $f'_{2i}$ for some $i$ (i.e.\ $T^*$ is a channel of $F$), orient the edges of $T^*$ by rooting it at $f'_{2i}$.
Otherwise, by Lemma \ref{tecc}, $T^*$ is a bay of $F$; consider the edge $\{u,z\}$, where $u$ is the unique vertex in $T^*$ adjacent to $z$ (cf.\ Lemma \ref{tecc}); then $T^*\cup\{u,z\}$ is a tree; call it an augmented bay, denote it by $\overline{T^*}$, and orient its edges by rooting it at $z$.
Then $\mu_F$ is defined to consist of the tail half-edges of the edges in $F$, the tail half-edges of the edges of the channels of $F$, and the tail half-edges of the edges of the augmented bays $\overline{T^*}$ of $F$.

One readily checks that $\mu_F$ is a perfect matching of $\mu_F$ of $\widehat{H}_{G}\setminus \{v'_1,f'_2,v'_3,\dotsc,v'_{2n+1}\}$. Indeed, these edges are clearly disjoint. Furthermore, each original vertex is matched because it is contained in the spanning forest $F$, each face-vertex is matched because it is contained in the dual spanning forest $F^*$, and each edge-vertex is matched because it is contained either in the edge of $G$ containing it, or in the edge of $G^*$ containing it. So the map $F\mapsto\mu_F$ is well-defined.

By construction, the maps $\mu\mapsto F_\mu$ and $F\mapsto\mu_F$ are inverses of each other. This completes the proof of part (a).

Part (b) follows by the same arguments. The only difference is that, since we are assuming that the perfect matching $\mu$ of $\widehat{H}_{G}\setminus \{v'_1,f'_2,v'_3,\dotsc,v'_{2n+1}\}$ contains the edges of the unique perfect matching of the path $P_i$, for all $i\in I$, the path $Q_i$ in the proof of part (a) will coincide with $P_i$, for all $i\in I$. By our construction of the bijection that proved (a), this is equivalent to the requirement that the corresponding spanning forest $F=F_\mu$ contains the odd indexed paths in $\{P_i:i\in I\}$, and its dual spanning forest $F^*$ contains the even indexed paths in $\{P_i:i\in I\}$. This completes the proof.
\end{proof}  

{\it Remark $2$.} The statement of Theorem \ref{tec} readily extends to the case when there is an arbitrary weight on the edges of the planar dual graph $G^*$, and the edges of $H_G$ contained in edges of $G^*$ are weighted by the weights of the latter. The only difference is that in this case the weight of a spanning forest $F$ of $G\setminus\{v_2,v_4,\dotsc,v_{2n},v'_2,v'_4,\dotsc,v'_{2n}\}$ needs to be defined to be the product of the weights of its constituent edges, times the product of the weights of the edges of $G^*$ making up $F^*$. Then the arguments in the proof go through without change.

\medskip
{\it Remark $3$.} Theorem \ref{tec} provides an alternative way to prove Theorem \ref{tea}. Indeed, if condition $(i)$ stated at the beginning of this section is also met, then Theorem \ref{tec} can also be applied with the roles of the vertices denoted by primed and unprimed symbols swapped. Then the sets of perfect matchings arising from these two applications of Theorem \ref{tec} are precisely the ones involved in the statement of Theorem \ref{tea}. On the other hand, the sets of spanning trees arising from the two applications of Theorem \ref{tec} only differ by the choice of their roots, so they are equinumerous.

\vspace{5mm}


\section{Independence results for spanning trees chosen uniformly at random} 

In this section we present more results concerning the equality of the number of spanning trees that satisfy certain conditions. The involved graphs are symmetric planar graphs. Like in the case of Theorem \ref{tbc}, the proofs will again be based on a bijection between the perfect matchings of certain related graphs --- namely, a refinement of the second author's bijection in \cite[Lemma1.1]{ciucu1997enumeration}.

The following elementary lemma will be useful in our proofs.

\begin{lemma}
\label{txx}
Let $G$ be a plane graph, $G^*$ its planar dual and $H_G$ the dual refinement of $G$ described at the beginning of Section $2$. Let $C$ be a cycle in $G$. Then the number of vertices of $H_G$ in the interior of $C$ is odd.

\end{lemma}

\begin{proof} We apply Euler's theorem to the subgraph $K$ of $H_G$ induced by the vertices on or inside~$C$. Let $v$ be the number of vertices of $G$ in the interior of $C$, $e$ the number of edges of $G$ in the interior\footnote{Note that endpoints of such an edge may be on $C$.} of $C$, and $f$ the number of faces of $G$ inside $C$. Let $n$ be the number of edges (of $G$) in $C$.

  The vertex set of the graph $K$ consists of $v+f+e$ vertices in the interior of $C$ and $2n$ vertices along $C$ ($n$ original vertices and $n$ edge-vertices). Each edge $e$ of $G$ in the interior of $C$ produces four edges of $K$, while each edge of $G$ in $C$ generates three edges of $K$, for a total of $4e+3n$ edges in $K$. The number of bounded faces of $K$ is obtained by adding up the number of vertices of each of the $f$ faces of $G$ inside $C$ --- equivalently, by adding up the number of {\it edges} of each of the $f$ faces of $G$ inside $C$; since edges in the interior of $C$ bound precisely two faces, this equals $2e+n$. Therefore, Euler's theorem for the graph $K$ gives
\begin{equation}  
(v+e+f+2n)-(4e+3n)+(2e+n)=(v+e+f)-2e=1. 
\label{eca}
\end{equation}  
Thus $v+e+f$, which is the number of vertices of $H_G$ in the interior of $C$, is odd.
\end{proof}  

Let $G$ be a weighted plane graph symmetric across the symmetry axis $\ell$ (i.e.\ $G$ comes equipped with a weight function on the edges which is constant on the orbits of the reflection across $\ell$); we consider $\ell$ to be horizontal. The next result --- which is the refinement of \cite[Lemma 1.1]{ciucu1997enumeration} mentioned above --- concerns perfect matchings of $G$, for whose existence a necessary condition is that $G$ has an even number of vertices on $\ell$. Assume therefore that this condition holds, and label the vertices of $G$ on $\ell$ from left to right $a_1,b_1,a_2,b_2,\dotsc,a_k,b_k$. Since the matching generating function\footnote{ The matching generating function of a weighted graph $G$ is the sum of the weights of the perfect matchings of~$G$, where the weight of a perfect matching is defined to be the product of the weights of its constituent edges. When all weights are 1 this becomes the number $\M(G)$ of perfect matchings of $G$.} is clearly multiplicative with respect to disjoint unions of graphs, we will henceforth also assume that all graphs under consideration are connected.

Given an edge $e$ of $G$, denote by $e'$ the reflection of $e$ across $\ell$.

\begin{thm} 
\label{tcb}
Let $G$ be a weighted symmetric plane graph, and denote its vertices on the symmetry axis, from left to right, by $a_1,b_1,a_2,b_2,\dotsc,a_k,b_k$. Let $E=\{e_1,\dotsc,e_s\}$ consist of $s\leq k$ pairwise disjoint edges of $G$, each incident to exactly one\footnote{ The assumption that there is no edge in $E$ such that both its endpoints are $a_i$'s can be made without loss of generality. Indeed, suppose there is such an edge $e$. Then its reflection $e'$, which is also and edge of $G$, has to be the same as $e$, as $G$ has no parallel edges. But then $e$ must be along $\ell$, and therefore has to contain at least one vertex~$b_i$, a contradiction.} of the $a_j$'s.
For each $I\subset[s]$, consider the subset of the set ${\mathcal M}(G)$ of perfect matchings of $G$ defined by
\begin{equation}
{\mathcal M}_E^I(G):=\{\mu\in{\mathcal M}(G): e_i\in\mu \text{\rm\ for\ } i\in I, e'_i\in\mu \text{\rm\ for\ } i\in [s]\setminus I\}.
\label{ecb}
\end{equation}
Then the sum of the weights of the perfect matchings in ${\mathcal M}_E^I(G)$ is the same for all $I\subset[s]$.

\end{thm}

Note that some of the edges in $E$ could be along the symmetry axis. If $e$ is such an edge, then~$e'$ is the same as $e$.

\begin{proof} 
  It is enough to prove the statement of the theorem for any two subsets $I,I'\subset[s]$ with $I=I'\cup\{i\}$, for some $i\in[s]\setminus I'$. Let $\mu\in{\mathcal M}_E^I(G)$, and let $\mu'$ be the reflection of $\mu$ across $\ell$. Then $\mu\cup\mu'$ is a disjoint union of cycles of $G$; let $C$ be the cycle containing
the vertex of $e_i$ on $\ell$ (which by assumption is one of the $a_j$'s).
By the argument in the proof of \cite[Lemma 1.1]{ciucu1997enumeration}, $C$ contains exactly one other vertex in $G\cap\ell$
, and that other vertex is a $b_j$.

 Deﬁne $\mu''$ to be the perfect matching of $G$ obtained from $\mu$ by replacing $\mu\cap C$ by $\mu'\cap C$. Then clearly $\mu''$ is an element of ${\mathcal M}_E^{I\setminus \{i\}}(G)$, and the correspondence $\mu\mapsto\mu''$ is a weight-preserving
involution between the elements of ${\mathcal M}_E^I(G)$ and those of ${\mathcal M}_E^{I'}(G)$.
\end{proof}  



\medskip
 {\it Remark $4$.} We point out that \cite[Lemma 1.1]{ciucu1997enumeration} follows from Theorem \ref{tcb} as a corollary. Indeed, partition the perfect matchings of the reduced subgraphs of \cite[Lemma 1.1]{ciucu1997enumeration} according to which edge~$e_i$ incident to $a_i$ they contain, for all $i=1,\dotsc,k$. Then the classes of the resulting partitions can be matched in pairs that are in bijective correspondence, by
 Theorem \ref{tcb}.

\medskip
\begin{thm} 
\label{tcc}
Let $G$ be a weighted plane graph symmetric about the symmetry axis $\ell$, and denote its vertices on $\ell$, from left to right, by $a_1,a_2,\dotsc,a_k$. Let $E=\{e_1,\dotsc,e_s\}$ consist of~$s\leq k$ pairwise disjoint edges of $G$, each $e_i$ being incident to exactly one of the\footnote{ In particular, no edge in $E$ is along $\ell$; however, some other edges of $G$ may be along $\ell$.} $a_j$'s. Orient each edge $e$ in $E$ so that it points away from $\ell$. Let $v$ be a vertex of $G$ that is both on the unbounded face and on $\ell$, and is not incident to any of the edges in $E$. 
For each $I\subset[s]$, consider the subset of the set~${\mathcal T}^v(G)$ of spanning trees of $G$ rooted at $v$ defined by\footnote{ For each oriented edge $e$, its reflection $e'$ across $\ell$ has the orientation induced by the reflection.}
\begin{equation}
{\mathcal T}_{E,I}^v(G):=\{T\in{\mathcal T}^v(G): {e_i}\in T \text{\rm\ for\ } i\in I, {e'_i}\in T \text{\rm\ for\ } i\in [s]\setminus I\}.
\label{ecc}
\end{equation}
$($In other words, ${\mathcal T}_{E,I}^v(G)$ is the set of all spanning trees $T$ of $G$ rooted at $v$ in which the unique edge exiting $a_i$ is $e_i$, for $i\in I$, and $e_i'$, for $i\in [s]\setminus I$.$)$
Then the sum of the weights of the spanning trees in ${\mathcal T}_{E,I}^v(G)$ is the same for all $I\subset[s]$.

\end{thm}

\begin{figure}[t]
\vskip0.2in
\centerline{
\hfill
{\includegraphics[width=0.5\textwidth]{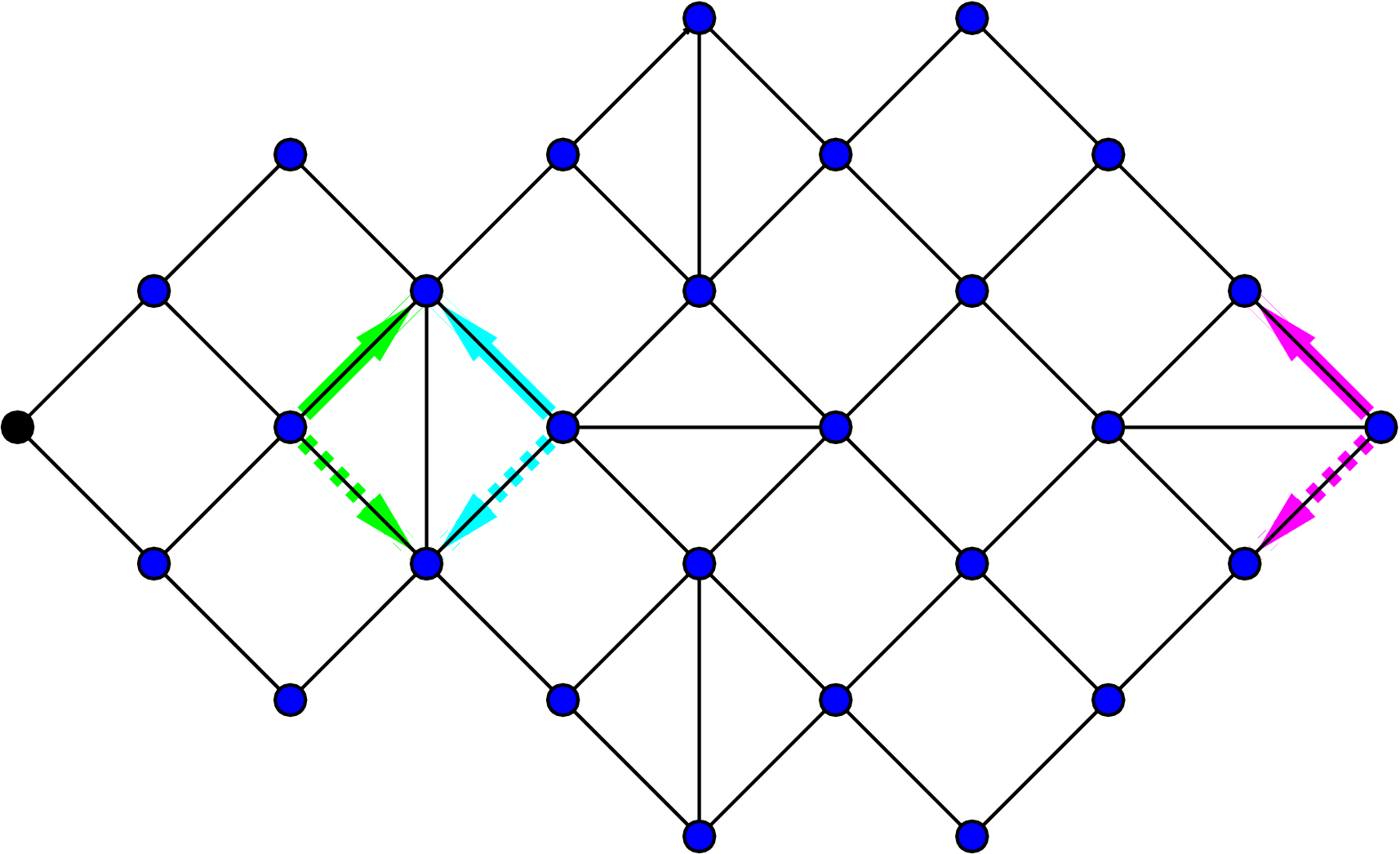}}
\hfill
}
\caption{\label{fca} An illustration of Theorem $\ref{tcc}$. Consider the eight subsets of the set of spanning trees $T$ rooted at the black vertex obtained by requiring that $T$ contains, from among the six shown oriented edges, one of each color (either the solid one or the dashed one). Then these eight subsets have the same number of elements.
}
\end{figure}

\begin{proof}


Consider the graph $H_G$ (the dual refinement of $G$ described at the beginning of Section 2), and draw it so that it is symmetric about $\ell$; in particular, if $f$ is a vertical edge of $G$ crossing $\ell$ --- which is an allowed possibility --- the edge-vertex of $H_G$ corresponding to $f$ will be on $\ell$. We claim that the number of vertices of $H_G$ on $\ell$ between $a_{i}$ and $a_{i+1}$ is odd.

Indeed, one of the following situations must occur: (1) vertices $a_i$ and $a_{i+1}$ are connected by an edge of $G$, (2) $a_i$ and $a_{i+1}$ are on the same face of $G$, but they are not connected by an edge, or (3) $a_i$ and $a_{i+1}$ belong to adjacent faces $F_1$ and $F_2$ of $G$, sharing a vertical edge $f$ that crosses $\ell$. In case (1), the edge-vertex corresponding to the edge connecting $a_i$ and $a_{i+1}$ is the only vertex of $H_G$ on $\ell$ between $a_{i}$ and $a_{i+1}$. In case (2), the face-vertex corresponding to the face containing $a_i$ and $a_{i+1}$ is the only  vertex of $H_G$ on $\ell$ between $a_{i}$ and $a_{i+1}$. In case (3), there are three vertices of $H_G$ on $\ell$ between $a_{i}$ and $a_{i+1}$: two face-vertices corresponding to the faces $F_1$ and $F_2$, and an edge-vertex corresponding to the vertical edge $f$.

This implies that for any subset $S$ of $\{a_1,\dotsc,a_k\}$, the number of vertices of $H_G$ on $\ell$ between two consecutive elements of $S$ is odd. In particular, the subset of $\{a_1,\dotsc,a_k\}$ incident to the edges in $E$ can be regarded as a subset of the set of $a_j$'s in Theorem \ref{tcb}.

Let $a_{j_1},\dotsc,a_{j_s}$ be the vertices of $e_1,\dotsc,e_s$ contained in $\ell$, respectively.
Using Temperley's bijection (see Theorem \ref{tda}(b)) and Corollary \ref{tdb}, the statement that the sum of the weights of the spanning trees in ${\mathcal T}_{E,I}^v(G)$ is the same for all $I\subset[s]$ is equivalent to the statement that the sum of all the weights in the set of perfect matchings of $H_G\setminus v$ in which each $a_{j_i}$ is matched by the half-edge of $e_i$ incident to it for $i\in I$, and by the half-edge of $e_i'$ incident to it for $i\in [s]\setminus I$, is independent of $I$. However, this follows directly from Theorem \ref{tcb}. 
\end{proof}




The next theorem was motivated by the following result due to Johansson. Let $Y_n$ be the indicator random variable for the event that the path from $(n,n)$ to infinity in the uniform spanning tree on the infinite square grid $\Z^2$ has its first step above the diagonal $x=y$. Then it follows from work of Johansson \cite{Johansson} on domino tilings that the $Y_n$'s are independent and identically distributed random variables with Bernoulli distribution of parameter $1/2$.

The following is a finitary generalization of this.

\begin{figure}[t]
\vskip0.2in
\centerline{
\hfill
{\includegraphics[width=0.5\textwidth]{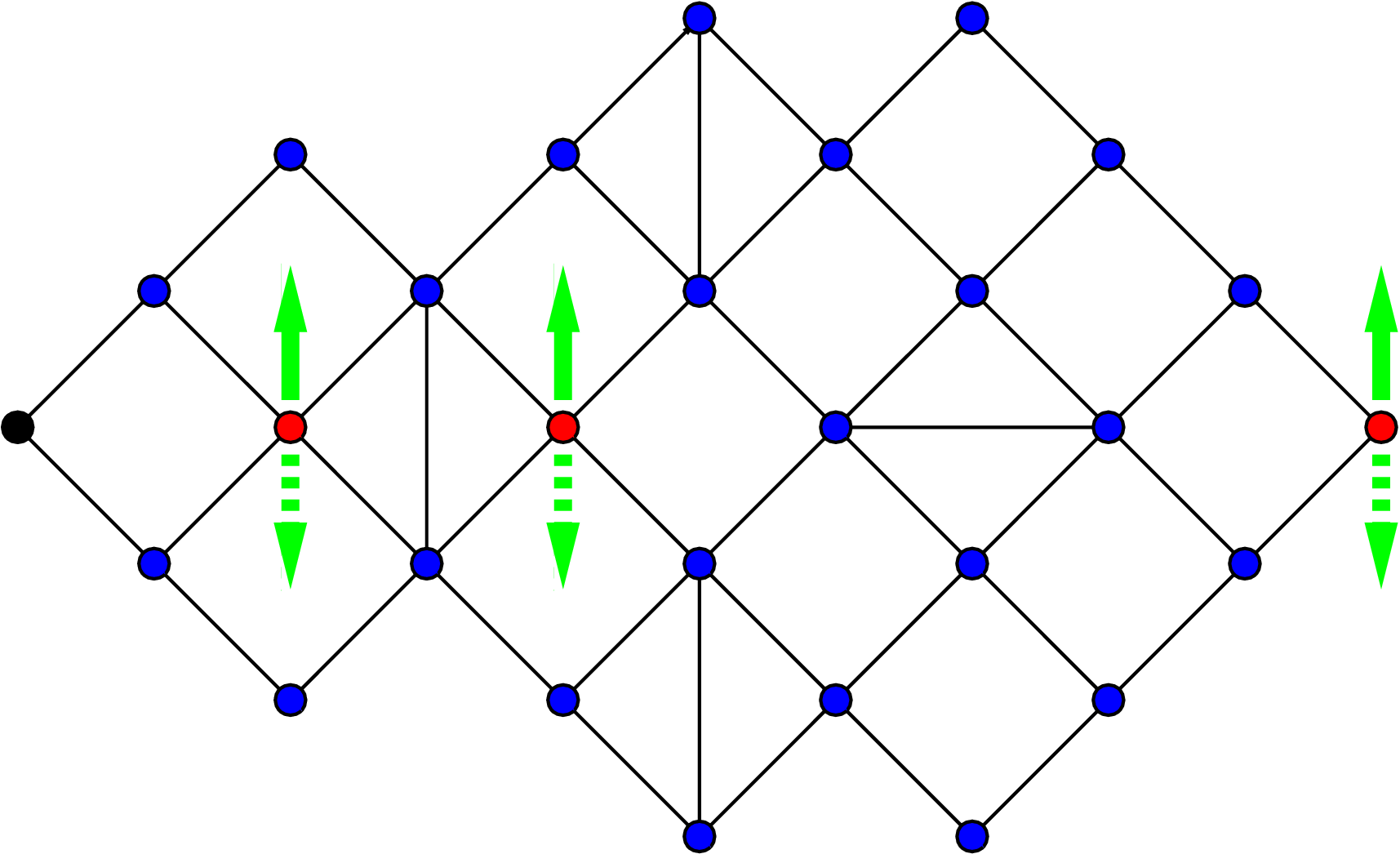}}
\hfill
}
\caption{\label{fcb} An illustration of Theorem $\ref{tcd}$. Select a spanning tree $T$ uniformly at random from the set of spanning trees of $G$ rooted at the black vertex $v$. Shown in red are the vertices different from $v$ on the symmetry axis which are not incident to an edge contained in the symmetry axis. For each red vertex define a random variable that is 1 if the vertex is exited in $T$ above the symmetry axis, and 0 if exited below. Then these are independent Bernoulli random variables of parameter~$1/2$.
}
\end{figure}

\begin{thm} 
\label{tcd}
Let $G$ be a weighted plane graph symmetric about the symmetry axis $\ell$, and let $v$ be a vertex of $G$ that is both on $\ell$ and on the unbounded face. 
Let $v_1,\dotsc,v_n$ be the vertices of $G\setminus v$ that are on $\ell$ and are not incident to edges contained in $\ell$.

Select a spanning tree $T$ uniformly at random from the set ${\mathcal T}^v(G)$ of spanning trees of $G$ rooted at $v$, and let $Y_i$ be the random variable which takes value $1$ if, in $T$, vertex $v_i$ is exited along an edge above\footnote{I.e.\ in $T$ the unique outgoing edge from vertex $v_i$ is an edge that is situated above $\ell$.} $\ell$, and value $0$ otherwise. Then $\{Y_i:i=1,\dotsc,n\}$ are independent and identically distributed random variables with Bernoulli distribution of parameter $1/2$.

\end{thm}

\begin{proof}
  Let $I\subset[n]$, and denote by ${\mathcal T}_I^v(G)$ the subset of ${\mathcal T}^v(G)$ consisting of all spanning trees in which $v_i$ is exited along an edge above $\ell$ for $i\in I$, and along an edge below $\ell$ for $i\in [n]\setminus I$. Then to prove the theorem we need to show that the total weight of the spanning trees in ${\mathcal T}_I^v(G)$ is the same for all $I\subset[n]$, and that each $Y_i$ is a Bernoulli random variable with parameter $1/2$.

Let $d_i$ be the number of edges of $G$ incident to $v_i$ from above, for $i=1,\dotsc,n$. Then each ${\mathcal T}_I^v(G)$ can be naturally partitioned into $d_1d_2\cdots d_n$ classes, according to which of the $d_i$ edges incident from above to $v_i$ (if $i\in I$) or which of the $d_i$ edges incident from below to $v_i$ (if $i\in [n]\setminus I$) is the edge along which vertex $v_i$ is exited, for $i=1,\dotsc,n$. For any two subsets $I,I'\subset[n]$, the $d_1\cdots d_n$ classes of the partition resulting from $I$ can be matched to the $d_1\cdots d_n$ classes of the partition resulting from $I'$ so that in corresponding classes each vertex $v_i$ is exited either along the same edge, or along two edges that are reflections of each other across $\ell$. By Theorem \ref{tcc}, the matched up classes of these partitions have the same total weight, and therefore ${\mathcal T}_I^v(G)$ and ${\mathcal T}_{I'}^v(G)$ themselves have the same total weight. This proves the first statement at the end of the previous paragraph.

To prove the second statement, note that the argument above works just as well if instead of considering the set $\{v_1,\dotsc,v_n\}$ of all the vertices of $G\setminus v$ that are on $\ell$ and are not incident to edges contained in $\ell$, we consider any subset of it. The case when this subset is just the singleton $\{v_i\}$ implies that $Y_i$ is a Bernoulli random variable with parameter $1/2$. This completes the proof. \end{proof}
  
The following special case leads to a counterpart of Johansson's above mentioned result which seems to be new.

\begin{figure}[t]
\vskip0.2in
\centerline{
\hfill
{\includegraphics[width=0.35\textwidth]{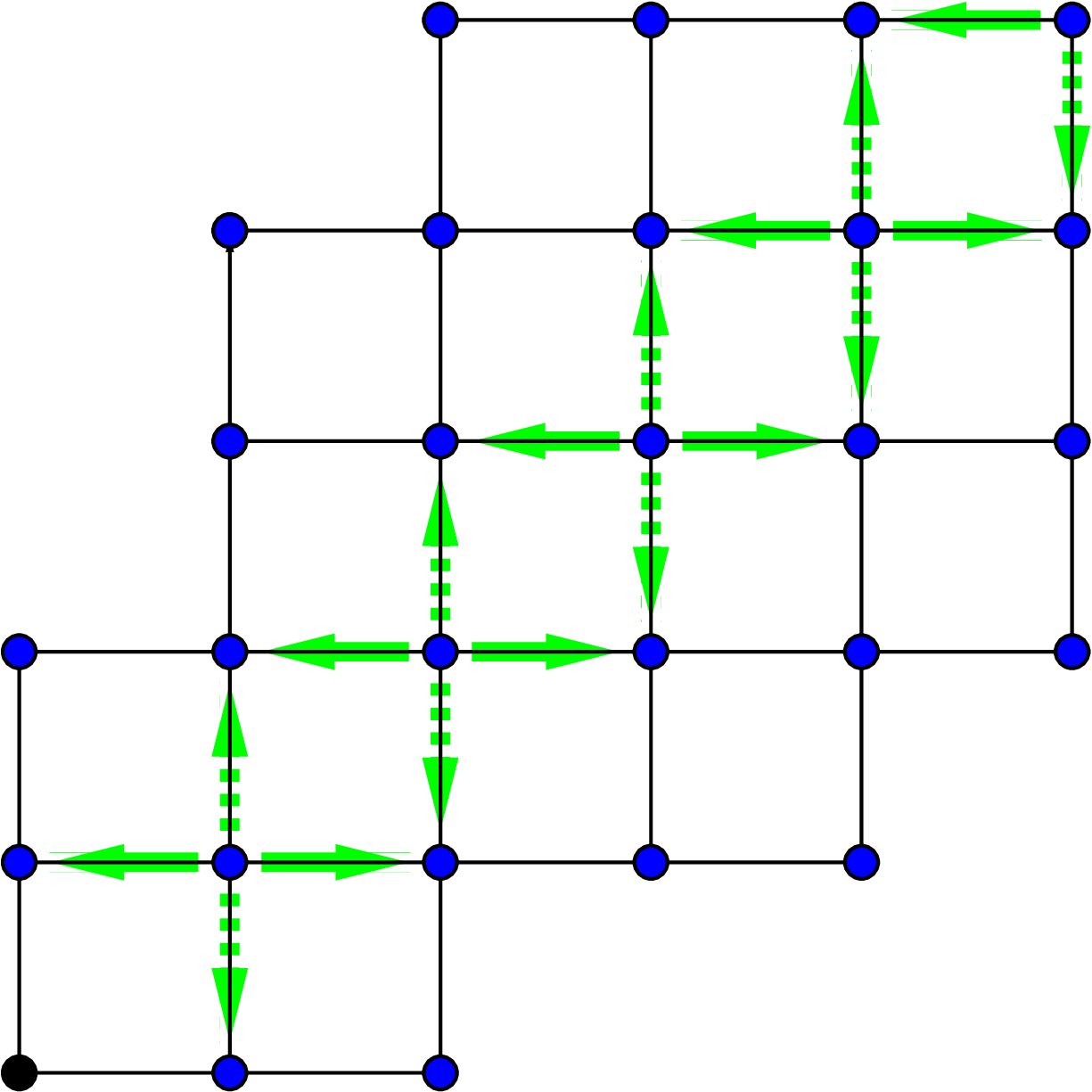}}
\hfill
}
\caption{\label{fcc} An illustration of Corollary $\ref{tce}$. Select a spanning tree $T$ uniformly at random from the set of spanning trees of $G$ rooted at the black vertex $v$. For each vertex different from $v$ on the symmetry axis define a random variable that is~1 if the vertex is exited in $T$ horizontally, and 0 if exited vertically. Then these are independent Bernoulli random variables of parameter~$1/2$.
}
\end{figure}

\begin{cor} 
\label{tce}
Let $G$ be a finite subgraph of the grid graph $\Z^2$, symmetric about a diagonal $\ell$ of the grid, and let $v$ be a vertex of $G$ that is both on $\ell$ and on the unbounded face.

Select a spanning tree $T$ uniformly at random from the set ${\mathcal T}^v(G)$ of spanning trees of $G$ rooted at $v$, and let $Z_i$ be the random variable which takes value $1$ if, in $T$, vertex $v_i$ is exited along a horizontal edge, and value $0$ otherwise. Then $\{Z_i:i=1,\dotsc,n\}$ are independent and identically distributed random variables with Bernoulli distribution of parameter $1/2$.

\end{cor}

\begin{proof}
  Note that for any edge $e$ incident to a vertex on $\ell$, $e$ is  horizontal (resp., vertical) precisely if the reflection $e'$ of $e$ across $\ell$ is vertical (resp., horizontal). The statement follows then by the same argument as in the proof of Theorem \ref{tcd}. Indeed, one just needs to change the definition of ${\mathcal T}_I^v$ to be the set of spanning trees $T$ of $G$ rooted at $v$ in which the unique outgoing edge from $v_i$ is horizontal, if $i\in I$, and vertical otherwise. The classes of the resulting partitions of the sets of spanning trees can be matched up in equinumerous pairs by Theorem \ref{tcc}. \end{proof}

{\it Remark $5$.} By taking $G$ to be the subgraph of $\Z^2$ induced by the lattice points $(x,y)$ with $\max(|x|,|y|)\leq m$, and taking the limit as $m$ goes to infinity, we obtain that if  $Z_n$ is the indicator random variable for the event that the path from $(n,n)$ to infinity in the uniform spanning tree on the infinite square grid $\Z^2$ has its first step along a horizontal edge, then the $Z_n$'s are independent and identically distributed random variables with Bernoulli distribution of parameter~$1/2$. This is a counterpart of Johansson's result that seems to be new.

\medskip
We end this section by presenting a combinatorial proof of a result of Lyons (see \cite[Theorem~4.8]{Lyons}) on the independence of edges along a diagonal being contained in the uniform spanning tree on $\Z^2$.

\medskip
\begin{thm}[\sc Lyons \cite{Lyons}]
\label{tcf}
Let $e$ be an edge in $\Z^2$. For $n\in\Z$, let $X_n$ be the indicator random variable of the event that the edge $e+(n,n)$ is contained in the uniform spanning tree on $\Z^2$. Then the $X_n$'s are independent Bernoulli random variables of parameter $1/2$.
\end{thm}  

\begin{figure}[t]
\vskip0.2in
\centerline{
\hfill
{\includegraphics[width=0.4\textwidth]{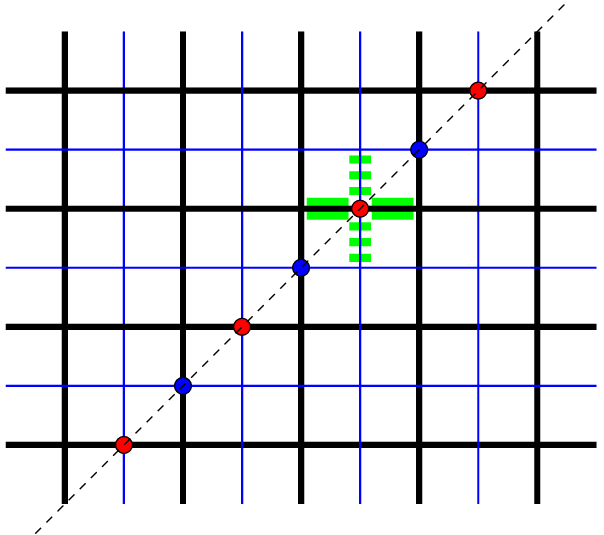}}
\hfill
}
\caption{\label{fcd} 
The graph $\Z^2$ is shown in black lines; the refinement $(\frac12\Z)^2$ is obtained by adding the blue lines. The lattice points $a_i$ are shown in red, the lattice points $b_i$ in blue. The three $T$ contains (resp., does not contain) the horizontal edge of $\Z^2$ with the solid green markings if and only if the corresponding perfect matching of $(\frac12\Z)^2$ contains one of the edges with a solid green marking (resp., one of the edges with the dashed green markings). Shifting along the cycle $C$ or path $P$ in the proof of Theorem \ref{tcf} changes the edge of the perfect matching incident to $a_i$ from horizontal to vertical, and from vertical to horizontal.
}
\end{figure}

\begin{proof}
  Let $0\leq s_1<\cdots<s_k$ be arbitrary integers and set $S=\{s_1,\dotsc,s_k\}$. Let $I\subset S$. Denote by ${\mathcal T}_I$ the set of spanning trees of $\Z^2$ rooted at infinity which contain the edges $\{(i,i),(i,i+1)\}$ for $i\in I$ and do not contain the edges $\{(i,i),(i,i+1)\}$ for $i\in S\setminus I$. To prove the theorem it suffices to show that there are natural bijections between the sets ${\mathcal T}_I$, for all $I\subset S$.

Note that Temperley's construction in Theorem \ref{tda} can be used to yield a natural bijection between spanning trees of $\Z^2$ rooted at infinity and perfect matchings of the refined grid graph~$(\frac12\Z)^2$ whose vertices form the set $\{(x/2,y/2):x,y\in\Z\}$. 

Denote by $a_i$ the lattice point $(i+\frac12,i)$, and by $b_i$ the lattice point $(i+1,i+\frac12)$, for $i\in\Z$. All these lattice points are on a diagonal $\ell$ of the lattice $(\frac12\Z)^2$, and the $a_i$'s and $b_i$'s alternate along $\ell$.

Consider a spanning tree $T$ of $\Z^2$ rooted at infinity. Then $T\in{\mathcal T}_I$ if and only if the perfect matching $\mu$ of $(\frac12\Z)^2$ corresponding to $T$ via Temperley's construction contains one of the two horizontal edges incident to  $a_i$ for $i\in I$, and one of the two vertical edges incident to $a_i$ for $i\in S\setminus I$. 
Let ${\mathcal M}_I$ be the set of perfect matchings of $(\frac12\Z)^2$ satisfying these conditions. 

The lattice points of $(\frac12\Z)^2$ on $\ell$ in the first quadrant are then, in order from southwest to northeast, $a_1,b_1,a_2,b_2,\dotsc,a_k,b_k,\dotsc$.
So the set-up is the same as in Theorem \ref{tcb}, with the one difference that instead of the finite graph $G$ we have the infinite graph $(\frac12\Z)^2$. Despite this difference, the construction in the proof of Theorem \ref{tcb} can still be used to obtain natural bijections between the sets ${\mathcal M}_I$, for all $I\subset S$; these in turn will induce, via Temperley's construction, natural bijections between the sets ${\mathcal T}_I$, for all $I\subset S$, thus proving the statement of the theorem.

  Indeed, it suffices to construct a bijection between ${\mathcal M}_I$ and ${\mathcal M}_{I\setminus t}$, for any $t\in I$. Let $\mu\in{\mathcal M}_I$, and let $\mu'$ be its reflection across $\ell$. Then $\mu\cup\mu'$ is a union of disjoint cycles {\it and infinite paths}. If~$a_t$ is contained in a cycle $C$ of $\mu\cup\mu'$, replace the part of $\mu$ along $C$ by $\mu'$; since, as we have seen in the proof of Theorem \ref{tcb}, $C$ only contains one other lattice point of $(\frac12\Z)^2$ on $\ell$, and that other lattice point is a $b_j$, the resulting perfect matching $\mu''$ is in ${\mathcal M}_{I\setminus t }$. If on the other hand $a_t$ is contained in an infinite path $P$ of $\mu\cup\mu'$, let $\mu''$ be the perfect matching of $(\frac12\Z)^2$ obtained from $\mu$ by replacing its portion along $P$ by $\mu'$. Then $\mu''$ is clearly in  ${\mathcal M}_{I\setminus t}$, $\mu\mapsto\mu''$ is a natural bijection between ${\mathcal M}_I$ and ${\mathcal M}_{I\setminus t}$, and the proof is complete.
\end{proof}

\section{Concluding Remarks}

In this paper we presented three settings which illustrate the usefulness of the strong connection between perfect matchings of plane graphs and spanning trees (or spanning forests) of some closely related graphs. First, Theorem \ref{tba}, a bijection between perfect matchings of two different families of graphs, found its most natural phrasing as a statement about spanning trees (see Theorem \ref{tbc}). Second, Theorem \ref{tec}, an extension of Temperley's classical bijection between spanning trees and perfect matchings, was used to provide a bijection between two families of graphs (see Corollary \ref{teb}), which solves an open problem posed by Corteel, Huang and Krattenthaler \cite[Page 3]{Corteeletal2023AztecT1}. And third, a perfect matching bijection for symmetric plane graphs (see \cite[Lemma1.1]{ciucu1997enumeration} and also its refinement Theorem \ref{tcb} in this paper), was used to prove independence results in uniform spanning trees (see Theorems \ref{tcd}, \ref{tce} and \ref{tcf}).

It would be interesting to find an explanation for the equality $\M(G^{+})=\M(G^{-})$ of Theorem \ref{tba} from the point of view of linear algebra --- in other words, a direct justification of the fact that the two Pfaffians which give the number of perfect matchings of the graphs $G^+$ and $G^-$ (cf.\ the work of Kasteleyn \cite{kasteleyn1961statistics} and Temperley and Fisher \cite{temperley1961dimer}) are equal.





\bigskip
{\bf Acknowledgments.} When the paper was completed, the first author was participating in the program ``Geometry, Statistical Mechanics, and Integrability” at Institute for Pure and Applied Mathematics (IPAM). He would like to thank the program organizers and IPAM staff members for their hospitality during his stay in Los Angeles. The second author thanks Russell Lyons for pointing out his result in Theorem 4.8 of \cite{Lyons}, for asking about the possibility of a combinatorial proof, and for interesting discussions on this subject. Theorems \ref{tcd} and \ref{tce}, as well as the new proof of Theorem \ref{tcf} presented in Section 6, grew out from these discussions.



\end{document}